\documentclass[12pt,a4paper,reqno]{amsart}
\usepackage{amssymb,a4wide,mathptmx,hyperref}
\providecommand{\BBb}[1]{{\mathbb{#1}}}

\providecommand{\cal}[1]{{\mathcal{#1}}}   

\newcommand{\ang}[1]{\langle#1\rangle}

\newcommand{\Bcirc}{\overset{\lower 1.5pt%
              \hbox{$@,@,@,@,@,\scriptscriptstyle\circ$}}B{}}
\newcommand{\Binfty}{\overset{\lower 1.5pt%
              \hbox{$@,@,@,@,@,\scriptscriptstyle\infty$}}B{}}
\newcommand{\bigdot}{\mathbin{\raise.65\jot\hbox{$\scriptscriptstyle\bullet$}}}

\newcommand{\B}{{\BBb B}}
\newcommand{\C}{{\BBb C}}

\newcommand{\dual}[2]{\langle\,#1,\,#2\,\rangle}

\newcommand{\erd}{\overset{\lower 1pt\hbox{\large.}}{e}
                  \overset{\lower 1pt\hbox{\large.}}{r}}

\newcommand{\Fcirc}{\overset{\lower 1.5pt%
               \hbox{$@,@,@,@,@,\scriptscriptstyle\circ$}}F{}}
\newcommand{\fracc}[2]{{
                \textstyle\frac{#1}{\raise 1pt\hbox{$\scriptstyle #2$}}}}

\newcommand{\fracnp}{\fracc np}

\newcommand{\fracci}[2]{{\frac{#1}{\raise 1pt\hbox{$\scriptscriptstyle #2$}}}}

\newcommand{\im}{\operatorname{i}}

\newcommand{\lap}{\operatorname{\Delta}}
\newcommand{\loc}{\operatorname{loc}}

\newcommand{\nrm}[2]{\|#1\|_{#2}}
\newcommand{\Nrm}[2]{\bigl\|#1\bigr\|_{#2}}
\newcommand{\norm}[2]{\mathinner{\|}#1\,|#2\|}

\newcommand{\op}[1]{\operatorname{#1}}
\newcommand{\OP}{\operatorname{OP}}
\newcommand{\N}{\BBb N}

\newcommand{\R}{{\BBb R}}
\newcommand{\Rn}{{\BBb R}^{n}}

\providecommand{\rom}[1]{\upn{#1}}

{
\end{minipage}%
\par\medskip\noindent}%
\newcounter{enmcount}\renewcommand{\theenmcount}{{\rm\arabic{enmcount}}}

\newcounter{rmcount}\renewcommand{\thermcount}{{\rm\roman{rmcount}}}
\newenvironment{rmlist}{%
\begin{list}{{\rm(\thermcount)}}{\setlength{\labelwidth}{\leftmargin}%
\usecounter{rmcount}}}{\end{list}}
\newcounter{Rmcount}\renewcommand{\theRmcount}{{\rm\Roman{Rmcount}}}

\newcommand{\scal}[2]{(\,#1\,|\, #2\,)}

\newcommand{\singsupp}{\operatorname{sing\,supp}}
\newcommand{\supp}{\operatorname{supp}}
\newcommand{\Z}{\BBb Z}

\renewcommand{\check}[1]{\overset{{\scriptscriptstyle \vee}}{#1}}
\renewcommand{\hat}[1]{\overset{{\scriptscriptstyle \wedge}}{#1}}

\renewcommand{\b}{\operatorname{b}}
\newcommand{\OPT}{\widetilde{\operatorname{OP}}}

\numberwithin{equation}{section}

\newtheorem{thm}{Theorem}
\numberwithin{thm}{section}
\newtheorem{prop}[thm]{Proposition}
\newtheorem{lem}[thm]{Lemma}
\newtheorem{cor}[thm]{Corollary}

\theoremstyle{definition}
\newtheorem{defn}[thm]{Definition}

\newtheorem{exmp}[thm]{Example}
 \numberwithin{exercise}{section}
 
\theoremstyle{remark}
\newtheorem{rem}[thm]{Remark}

\title[Type $1,1$-operators]{Type 1,1-operators defined 
by\\ vanishing frequency modulation}
\author{Jon Johnsen}
\address{Department of Mathematical Sciences, Aalborg University, 
Fredrik Bajers Vej 7G, DK-9220 Aalborg {\O}st, Denmark}
\keywords{Exotic pseudo-differential operators, type $1,1$, pseudo-local,
spectral support rule, regular convergence, flipped wavefront sets%
\\[4\jot] {\tt Appeared in 
"New developments in pseudo-differential operators"
      (L.~Rodino, M.~W.~Wong) Birkh{\"a}user 2008. 
      Operator Theory: Advances and Applications, Vol.~189, pp.~201--246.}%
}
\begin{document}
 \begin{abstract}
This paper presents a general definition of pseudo-differential operators of
type $1,1$; the definition is shown to be the largest one that is both 
compatible with negligible operators and stable under vanishing frequency
modulation. Elaborating counter-examples of Ching, H{\"o}rmander and 
Parenti--Rodino, 
type $1,1$-operators with unclosable graphs are proved to exist; 
others are shown to lack the microlocal property as they flip the wavefront
set of an almost nowhere differentiable function. 
In contrast the definition is shown to imply the pseudo-local 
property, so type $1,1$-operators cannot create singularities but only
change their nature. The familiar rule
 that the support of the argument is transported by the support of 
the distribution kernel is generalised to arbitrary type $1,1$-operators. 
A similar spectral support rule is also proved. 
As no restrictions appear for classical type $1,0$-operators, 
this is a new result which in many cases makes it unnecessary to
reduce to elementary symbols. As an important tool, a convergent sequence of
distributions is said to converge regularly if it moreover converges as
smooth functions outside the singular support of the limit. This notion is
shown to allow limit processes in extended versions 
of the formula relating operators and kernels.
 \end{abstract}
\maketitle
\section{Introduction}   \label{intr-sect}
Pseudo-differential operators are generally well understood
as a result of extensive analysis since the mid 1960s;
but there is an exception for operators of type $1,1$.
These have symbols in the H{\"o}rmander class $S^d_{1,1}(\Rn\times\Rn)$,
which is sometimes called exotic because of the operators' atypical
properties.

Recall that a symbol $a(x,\eta)\in C^{\infty}(\R^{2n})$ belongs to 
$S^{d}_{1,1}(\Rn\times \Rn)$ if it for all multiindices $\alpha$, $\beta$
satisfies the estimates
\begin{equation}
  |D^\alpha_\eta D^\beta_x a(x,\eta)|\le C_{\alpha,\beta}
  (1+|\eta|)^{d-|\alpha|+|\beta|}.
\end{equation}
For such a symbol, $a(x,D)u=\OP(a)u=Au$ is defined at least for $u$
in the Schwartz space $\cal S(\Rn)$ by the usual integral,
whereby 
$\cal Fu(\xi)=\hat u(\xi)=\int_{\Rn}e^{-\im x\cdot\xi} u(x)\,dx$
denotes the Fourier transformation,
\begin{equation}
  a(x,D)u(x)=(2\pi)^{-n}\int_{\Rn} e^{\im \xi\cdot\eta}
                a(x,\eta)\cal F u(\eta)\,d\eta.
  \label{auint-eq}
\end{equation}
The purpose of the present article is to suggest a \emph{general} definition of
operators with type $1,1$-symbols; that is, to define $a(x,D)u$ for
$u$ in a maximal subspace $D(A)$ such that 
\begin{equation}
  \cal S(\Rn)\subset D(A)\subset\cal S'(\Rn).  
  \label{DA-eq}
\end{equation}
Seemingly this question has not been addressed directly before.
But as a fundamental contribution,
L.~H{\"o}rmander \cite{H88,H89} used $H^s$-estimates 
to extend type $1,1$-operators by continuity from $\cal S(\Rn)$
and characterised the possible $s$ up to a limit point.

For other questions it seems necessary to have an explicit definition 
of type $1,1$-operators. Consider eg the pseudo-local property,
\begin{equation}
  \singsupp Au\subset \singsupp u\quad\text{for all}\quad u\in  D(A).
  \label{psdlocal-eq}
\end{equation}
In the proof of this, it is of course of little use just to know the
action of $A$ on $u\in \cal S(\Rn)$, as both sets are 
empty for such $u$. And to apply the fact that the distribution
kernel $K(x,y)$ of $A$ is $C^\infty$  for $x\ne y$
one would have to know more on $A$ and its domain $D(A)$
than just \eqref{DA-eq}.

To give a brief account of the present contribution, let $\psi\in
C^\infty_0(\Rn)$ denote an auxiliary function for which $\psi=1$ in a
neighbourhood of the origin. 
Then the frequency modulated versions of $u\in \cal
S'(\Rn)$ and of $a(x,\eta)$ with respect to $x$ are given for $m\in \N$ by
\begin{align}
  u^m &=\psi(2^{-m}D)u=\cal F^{-1}(\psi(2^{-m}\cdot)\cal Fu)
  \label{um-id}  \\
  a^m &=\psi(2^{-m}D_x)a=\cal F^{-1}_{\xi\to x}(\psi(2^{-m}\xi)\cal
F_{x\to\xi} a(\xi,\eta)).
  \label{am-id}
\end{align}
Therefore $a(x,D)$ is said to be \emph{stable} under 
\emph{vanishing} frequency modulation if for every $u$ in its domain
\begin{equation}
  a^m(x,D)u^m \xrightarrow[m\to\infty]{~} a(x,D)u 
 \quad\text{in}\quad \cal D'(\Rn).
  \label{stabl-id}
\end{equation}
Whilst classical pseudo-differential operators have this property, the
purpose is to show that \eqref{stabl-id} can be used as a 
definition of $a(x,D)u$ when $a\in S^\infty_{1,1}(\Rn\times\Rn)$ is given;
hereby $D(a(x,D))$ consists of the $u\in \cal S'(\Rn)$ for which
the limit exists independently of $\psi$. 
The limit in \eqref{stabl-id} serves as a substitute of 
the usual extensions by continuity from $\cal S(\Rn)$. 

In this introduction it is to be understood in \eqref{stabl-id} that, 
for all $u\in \cal S'(\Rn)$, 
\begin{equation}
  a^m(x,D)u^m=\OP(a^m(x,\eta)\psi(2^{-m}\eta))u,  
  \label{amum-eq}
\end{equation}
where  the right-hand side is in $\OP(S^{-\infty})$.
The expression $a^m(x,D)u^m$ itself is brief, but problematic if
taken literally since also $a^m(x,\eta)\in S^\infty_{1,1}$.
However, using that $\supp\cal F(u^m)\Subset\Rn$, it will later be seen 
that $a^m(x,D)u^m$ can be defined via \eqref{amum-eq} and that this is
compatible with \eqref{stabl-id}; thenceforth $a^m(x,D)u^m$ will be a short
and safe notation.

The definition is discussed in detail below, and shown to imply that 
type $1,1$-operators are pseudo-local 
(cf \eqref{psdlocal-eq} and Theorem~\ref{psdloc-thm}).
In comparison they do not in general preserve wavefront sets, 
for following C.~Parenti and L.~Rodino~\cite{PaRo78} 
a version of a well-known example due to C.~H.~Ching 
is shown to flip 
the wavefront set $\op{WF}(w_\theta)=\Rn\times(\R_+\theta)$ 
into $\Rn\times(\R_+(-\theta))$ for
some $w_\theta$, that when the order $d\in \,]0,1]$
is an almost nowhere differentiable function.
 
Moreover the following well-known support rule
is extended to arbitrary $a(x,D)\in
\OP(S^\infty_{1,1})$ with distribution kernel $K$
(cf Theorem~\ref{supprule11-thm}),
\begin{equation}
  \supp a(x,D)u\subset\overline{\supp K\circ\supp u}
  \quad\text{for all}\quad u\in D(a(x,D)).
  \label{Ka-eq}
\end{equation}
Here $\supp K\circ\supp u:= \bigl\{\,x\in \Rn \bigm| 
\exists y\in \supp u\colon (x,y)\in \supp K\,\bigr\}$, whereby $\supp K$ is
thought of as a relation on $\Rn$ that maps, or transports, every set
$M\subset \Rn$ to the set $(\supp K)\circ M$ of everything related to an
element of $M$.

There is an analogous result which seems to be new, 
even for classical symbols $a\in S^\infty_{1,0}$.
It gives a \emph{spectral} support rule, 
relating frequencies $\xi\in \supp\cal F(Au)$ to those in 
$\supp\cal Fu$:
if only $u\in  D(A)$ is such that 
\eqref{stabl-id} holds in the topology of $\cal S'(\Rn)$, 
then (cf Theorem~\ref{supp-thm})
\begin{gather}
   \supp\cal F(a(x,D)u)\subset\overline{\Xi},
  \label{Xi-eq}
 \\
  \Xi=\bigl\{\,\xi+\eta \bigm| (\xi,\eta)\in \supp\cal F_{x\to\xi} a,\ 
     \eta\in \supp\cal F u \,\bigr\}.
  \label{Xi-id}
\end{gather}
This is highly analogous to \eqref{Ka-eq}, for 
$\Xi=\supp\cal K\circ\supp \cal Fu$, where $\cal K$ is the kernel of
the conjugated operator $\cal F a(x,D)\cal F^{-1}$.
There is a forerunner of \eqref{Xi-eq}--\eqref{Xi-id} 
in \cite{JJ05DTL}, where it was only possible to 
cover the case $\cal F u\in \cal E'(\Rn)$, as the information on $D(a(x,D))$
was inadequate without the definition in \eqref{stabl-id}.

The spectral support rule \eqref{Xi-eq}
often makes it possible to by-pass a reduction to elementary symbols,
that were introduced by  
R.~Coifman and Y.~Meyer~\cite{CoMe78} in order to
control spectra like $\supp \cal F a(x,D)u$ in the $L_p$-theory of
general pseudo-differential operators.
Use of \eqref{Xi-eq}--\eqref{Xi-id} simplifies the theory, for it
would be rather inconvenient to add in \eqref{stabl-id} an extra limit process
resulting from approximation of $a(x,\eta)$ by elementary symbols.

Both \eqref{Ka-eq} and \eqref{Xi-eq} are established as consequences of
the formula relating an operator $A$ to its kernel 
$K\in \cal D'(\Rn\times\Rn)$,
\begin{equation}
  \dual{Au}{v}=\dual{K}{v\otimes u}.
  \label{AKuv-id}
\end{equation}
It is shown below (cf Theorems~\ref{AK-thm} and \ref{supprule11-thm})
that also the right-hand side makes sense as it stands 
for $u\in \cal D'(\Rn)$, although $K$ and $v\otimes u$
are distributions then, as long as $v$ is a test function such that
\begin{equation}
  \supp K\bigcap \supp v\otimes u\Subset\Rn\times\Rn, \qquad
  \singsupp K\bigcap \singsupp v\otimes u=\emptyset.
  \label{ssuppK-eq}
\end{equation}
That \eqref{ssuppK-eq} suffices for \eqref{AKuv-id} follows from the
extendability of the bilinear 
form $\dual{\cdot}{\cdot}$ in distribution theory to pairs
$(u,f)$ fulfilling analogous conditions. This simple extension of
$\dual{u}{f}$ has the advantage that $\dual{u}{f^\nu}\to\dual{u}{f}$ when 
$u$ or $f$ has compact support and
$f^\nu\in C^\infty(\Rn)$ are such that 
\begin{equation}
  f^\nu\xrightarrow[\nu\to\infty ]{~}f 
\quad\text{both in $\cal D'(\Rn)$ and in $C^\infty(\Rn\setminus\singsupp f)$}. 
\end{equation}
Such sequences $(f^\nu)$ are below said to converge
\emph{regularly} to $f$; they are easily obtained by convolution. 
In these terms, $\dual{\cdot}{\cdot}$ is
stable under regular convergence if one entry is in $\cal E'$. 

This set-up is convenient for the derivation of
\eqref{AKuv-id}--\eqref{ssuppK-eq} for type $1,1$-operators. 
Indeed, the kernel $K_m$ of the approximating operator
$a^m(x,D)u^m$ equals $K*\cal F^{-1}(\psi_m\otimes\psi_m)$ 
conjugated by the coordinate change $(x,y)\mapsto (x,x-y)$, 
so that $K_m$ converges regularly to $K$; 
whence \eqref{AKuv-id} results in the limit $m\to\infty $.
Based on this the support rules \eqref{Ka-eq}--\eqref{Xi-eq} follow in a
natural way.

However, the simple criterion in \eqref{ssuppK-eq} and its stability under
regular convergence, that might be known,
could be useful also for other questions. 

\smallskip

The main contributions in this paper consist first of all of the definition
\eqref{stabl-id} and the spectral support rule
\eqref{Xi-eq} ff; secondly of the proofs of pseudo-locality 
\eqref{psdlocal-eq} and the support rule \eqref{Ka-eq} as well as
the extension of the kernel formula
\eqref{AKuv-id}--\eqref{ssuppK-eq}. Moreover, $a(x,D)u$ is shown to be
compatible with the usual pseudo-differential operators 
(cf Sections~\ref{basic-sect}--\ref{gdef-sect}).

In addition there are various improvements of known results 
on type $1,1$-operators. This overlap is elucidated  
(in parenthetic remarks)
in the next section.

\subsection{On known results for type $1,1$-operators}   \label{history-ssect}
The pathologies of type $1,1$-operators were revealed around 1972--73.
On the one hand, C.~H.~Ching \cite{Chi72} gave examples of symbols 
$a\in S^0_{1,1}$ for which
the corresponding operators are unbounded from $L^2(\Rn)$ to $L^2(K)$ for
every $K\Subset\Rn$ (they can moreover be taken \emph{unclosable} in 
$\cal S'(\Rn)$, as shown in Lemma~\ref{cex-lem} below).
On the other hand, E.~M.~Stein (1972-73) showed
$C^s$-boundedness%
\footnote{Noted by Y.~Meyer \cite{Mey80},
with reference to lecture notes at Princeton
1972/73. E.~M.~Stein stated the $C^s$-result in \cite[VII.1.3]{Ste93};
at the end of Ch.~VII its origins were given as ``Stein [1973a]''
(that is
\emph{Singular integrals and estimates for the
Cauchy-Riemann equations}, Bull. Amer. Math. Soc. 79 (1973), 440--445)
but probably should have been
``Stein [1973b]'':
``\emph{Pseudo-differential operators}, Notes by D.H. Phong for a
course given at Princeton University 1972-73''.
}
for $s>0$ and orders $d=0$.

Afterwards C.~Parenti and L.~Rodino \cite{PaRo78} discovered that some type
$1,1$-operators do not preserve wavefront sets 
(cf Section~\ref{WF-ssect} 
where  this result of \cite{PaRo78} is extended to all $d\in \R$, $n\in \N$).
The pseudo-local property of type $1,1$-operators was also claimed in
\cite{PaRo78}, but not backed up  by adequate arguments; 
cf Remark~\ref{psdloc-rem} below.
(The question is therefore taken up in Theorem~\ref{psdloc-thm}, where
the first full proof is given.)

Around 1980, Y.~Meyer \cite{Mey80,Mey81}
obtained the famous property that a composition
operator $u\mapsto F(u)$, for a fixed $C^\infty$-function $F$ with 
$F(0)=0$, acting on $u\in H^s_{p}(\Rn)$ for $s>n/p$, can be
written 
\begin{equation}
  F(u)=a_u(x,D)u  
\end{equation}
for a specific $u$-dependent symbol $a_u\in S^0_{1,1}$. Namely, when
$1=\sum_{j=0}^\infty \Phi_j$ is a Littlewood--Paley partition of unity,
\begin{equation}
  a_u(x,\eta)=\sum_{j=0}^\infty m_j(x)\Phi_j(\eta),
\qquad
  m_j(x)=\int_0^1 F'(\sum_{k<j}\Phi_k(D)u(x)+t\Phi_j(D)u(x))\,dt.
\end{equation}
This gave a convenient proof of the fact that $u\mapsto F(u)$ maps
$H^s_p(\Rn)$ into itself for $s>n/p$. Indeed, this follows as
Y.~Meyer for general $a\in S^d_{1,1}$, 
using reduction to elementary symbols,
established continuity 
\begin{equation}
 H^{s+d}_p(\Rn)\xrightarrow[]{a(x,D)} 
 H^s_p(\Rn)\qquad\text{for $s>0$, $1<p<\infty$}. 
  \label{Meyer-eq}
\end{equation}
(In Section~\ref{composite-ssect} 
these results are deduced from
the definition in \eqref{stabl-id}, and continuity on $H^s_p$ of 
$u\mapsto F\circ u$ is added in a straightforward way
in Theorem~\ref{composite-thm}.)
It was also realised then that type $1,1$-operators show up in J.-M.~Bony's 
paradifferential calculus \cite{Bon}
of non-linear partial differential equations.

In the wake of this, T.~Runst \cite{Run85ex} treated the continuity 
in Besov spaces $B^{s}_{p,q}$ for $p\in \,]0,\infty]$ 
and in Lizorkin--Triebel spaces $F^{s}_{p,q}$ for $p\in \,]0,\infty[\,$,
although the necessary control of the frequency changes created by $a(x,D)$
was not quite achieved in \cite{Run85ex}.
(This flaw was explained and remedied in \cite{JJ05DTL}
by means of a less general version of \eqref{Xi-eq}.)

Around the same time G.~Bourdaud proved that 
a type $1,1$-operator
$a(x,D)\colon C^\infty_0(\Rn)\to\cal D'(\Rn)$ 
of order $0$
is $L_2$-bounded  
if and only if its adjoint $a(x,D)^*\colon C^\infty_0(\Rn)\to\cal D'(\Rn)$ 
is also a type $1,1$-operator; cf \cite{Bou83}, \cite[Th~3]{Bou88}. 

Except for a limit point,  L.~H{\"o}rmander 
characterised the $s\in \R$ for which a given $a\in S^d_{1,1}$ is bounded
$H^{s+d}\to H^s$; cf \cite{H88,H89} and also \cite{H97} where 
a few improvements are added. 
As a novelty in the analysis, an important role was shown to be
played by the twisted diagonal
\begin{equation}
  \cal T=\{\,(\xi,\eta)\in \Rn\times\Rn\mid \xi+\eta=0\,\}.
\end{equation} 
Eg, if the partially Fourier transformed symbol
$\hat a(\xi,\eta):=\cal F_{x\to\xi}a(x,\eta)$ 
vanishes in a conical neighbourhood of a non-compact part of $\cal T$, ie if 
\begin{equation}
  \exists C\ge1\colon C(|\xi+\eta|+1)< |\eta|\implies \hat a(x,\eta)=0,
  \label{TDC-cnd}
\end{equation}
then $a(x,D)\colon H^{s+d}\to H^s$ is continuous for
every $s\in \R$. Moreover, continuity for all $s>s_0$ was shown to be
equivalent to a specific asymptotic behaviour of 
$\hat a(\xi,\eta)$ at $\cal T$.  
For operators  with additional properties,
a symbolic calculus was also developed together with a sharp G{\aa}rding
inequality; cf \cite{H88,H89,H97}. 

For \emph{domains} of type $1,1$-operators, the scale $F^{s}_{p,q}(\Rn)$
of Lizorkin--Triebel
spaces  was recently shown to play a role, for
it was proved in \cite{JJ04DCR,JJ05DTL} that for all $p\in
[1,\infty[\,$, every $a\in S^d_{1,1}$ gives a bounded linear map 
\begin{equation}
  F^d_{p,1}(\Rn)\xrightarrow[]{a(x,D)} L_p(\Rn).
  \label{Fdp1-eq}
\end{equation}
This is a substitute of boundedness from
$H^d_p$ (or of $L_p$-boundedness for $d=0$), as 
$H^s_p=F^s_{p,2}\supsetneq F^s_{p,1}$ for $1<p<\infty$.
Inside the $F^{s}_{p,q}$ and $B^{s}_{p,q}$ scales,
\eqref{Fdp1-eq} gives maximal domains for
$a(x,D)$ in $L_p$, for it was noted in \cite[Lem.~2.3]{JJ05DTL} 
that already Ching's 
operator is discontinuous from $F^d_{p,q}$ to $\cal D'$ 
and from $B^{d}_{p,q}$ to $\cal D'$ for every $q>1$. 
Continuity was proved in \cite{JJ05DTL} for $s>\max(0,\fracnp-n)$,
$0<p<\infty$, as a map
\begin{equation}
  F^{s+d}_{p,q}(\Rn)  \xrightarrow[]{a(x,D)} 
  F^s_{p,r}(\Rn)   \quad\text{for}\quad r\ge q,\,
  r>\tfrac{n}{n+s}.
  \label{Fspq-eq}
\end{equation}
Moreover, \eqref{TDC-cnd} was shown to imply \eqref{Fspq-eq} 
for every $s\in \R$, $r=q$. Analogously for $B^{s}_{p,q}$.
(In Section~\ref{LP-ssect} it is shown how the techniques behind
\eqref{Fspq-eq} apply in the present set-up, and as a special case
\eqref{Meyer-eq} is rederived in this way; cf Theorem~\ref{Hsp-thm}.)

As indicated, a general definition of $a(x,D)u$ for a given symbol $a\in
S^d_{1,1}(\Rn\times\Rn)$ seems to have been unavailable hitherto.
L.~H{\"o}rmander \cite{H88,H89} estimated $Au$ for arbitrary 
$u\in \cal S(\Rn)$
in the $H^s$-scale, which of course gives a uniquely defined bounded
operator $A\colon H^{s+d}\to H^{s}$; and an extension of $A$ to 
$\bigcup_{s>s_0} H^{s+d}(\Rn)$ for some limit $s_0$ or possibly even
$s_0=-\infty$, depending on $a$.

R.~Torres \cite{Tor90} also estimated $Au$ for $u\in \cal S(\Rn)$,
using the framework of M.~Frazier and B.~Jawerth \cite{FJ1,FJ2}.
This gave unique extensions by continuity to maps $F^{s+d}_{p,q}(\Rn)\to
F^{s}_{p,q}(\Rn)$ for all $s$ so large that, for all multiindices $\gamma$,
\begin{equation}
  0\le|\gamma|<\max(0,\frac np-n,\frac nq-n)-s
  \implies A^*(x^\gamma)=0.
\end{equation}
(As noted in \cite{Tor90}, this is related to the conditions 
imposed at
the twisted diagonal $\cal T$ in the works of 
L.~H{\"o}rmander.) This approach will at most
define $A$ on $\bigcup F^{s}_{p,q}(\Rn)$. 

In addition it was shown in \cite[Prop.~1]{JJ05DTL} that every type
$1,1$-operator $A$ extends to the space $\cal F^{-1}\cal E'(\Rn)$.
(Extension to $\cal F^{-1}\cal E'$ is also considered in
Section~\ref{basic-sect} in connection with compatibility questions.)
Clearly $\cal F^{-1}\cal E'$ contains all polynomials 
$\sum_{|\alpha|\le k}c_\alpha x^\alpha$,
and these do not 
belong to $\bigcup H^s$, nor to $\bigcup F^{s}_{p,q}$, 
so this development only emphasises the need for a
general definition of type $1,1$-operators, without reference to spaces
other than $\cal S'(\Rn)$.

\subsection{Remarks on the construction}   \label{construction-ssect}
As indicated above, the extension of an operator $a(x,D)$ of type $1,1$ from
the Schwartz space 
$\cal S(\Rn)$ to a larger domain $D(a(x,D))$ in $\cal S'(\Rn)$ 
can roughly be made as follows:

Introducing $a^m(x,\eta)=\cal F^{-1}_{\xi\to x}(\hat
a(\xi,\eta)\psi_m(\xi))$, $\psi_m=\psi(2^{-m}\cdot)$
for a cut-off function $\psi\in C^\infty_0(\Rn)$ with $\psi=1$
around the origin, then $a(x,D)u$ is defined 
when $u\in \cal S'(\Rn)$ is such that
$a_{\psi}(x,D)u=\lim_{m\to\infty }\OP(a^m(x,\eta)\psi_m(\eta))$ 
exists in $\cal D'(\Rn)$ and does not depend on $\psi$. 
And in the affirmative
case, 
\begin{equation}
  a(x,D)u:=a_{\psi}(x,D)u=\lim_{m\to\infty} \OP(a^m(x,\eta)\psi_m(\eta))u.
  \label{aulim-eq}
\end{equation}
Fundamentally, the role of $a^m(x,\eta)$ is to
make the domain of $a(x,D)$ as large as possible: 
since $a(x,\eta)$ is less special than $a^m(x,\eta)$,
the demands on the pair $(a,u)$ would be stronger if only the
$\OP(a(x,\eta)\psi_m(\eta))u$ were required to converge. 
And the domain of $a(x,D)$ would possibly also be
smaller, had not the same $\psi$ been used twice to form 
$a^m(x,\eta)\psi_m(\eta)$. Finally, to take the limit in $\cal
S'(\Rn)$ instead might also exclude some $u$ from $D(a(x,D))$.
(However, the $\cal D'$-limit makes it more demanding to justify
compositions $b(x,D)a(x,D)$ of type $1,1$-operators.)

Although \eqref{aulim-eq} is an unconventional definition, it is not
as arbitrary as it may seem. In fact, cf
Theorem~\ref{uniq-thm} below, the resulting map $a\mapsto a(x,D)$,
$a\in S^\infty_{1,1}$, can be characterised as the \emph{largest} extension
of \eqref{auint-eq} that both gives operators 
\emph{stable} under vanishing frequency modulation and 
is \emph{compatible} with $\OP$ on $S^{-\infty}$. 
For $\delta<\rho$ it is even
compatible with the classes $\OP(S^\infty_{\rho,\delta})$
 in a certain local sense, termed strong compatibility below.

In addition to this, there are at least three simple indications that the
definition is reasonable. First of all, if the symbol $a(x,\eta)$ is
classical, say $a\in S^\infty_{1,0}$, then the usual $\cal
S'$-continuous extension of $\OP(a)$ fulfils
$\OP(a^m(x,\eta)\psi^m(\eta))u\to \OP(a)u$ as a consequence of standard
facts (cf Proposition~\ref{Stan-prop} below).

Secondly, the definition also gives back the usual product $au$, when $a(x)$
is a symbol in $S^\infty_{1,1}$ independent of $\eta$. In fact $a\in
C^\infty_{\b}(\Rn)$ then, and since $a^m(x)\psi_m(\eta)\in
S^{-\infty}(\Rn\times\Rn)$, every $u\in \cal S'$ gives the following
\begin{equation}
  \OP(a^m(x)\psi_m(\eta))u= a^m\cdot\cal F^{-1}(\psi_m\hat u)=a^m u^m
  \xrightarrow[m\to\infty]{~} au.
  \label{aulim'-eq}
\end{equation}
So despite the apparent asymmetry in $\OP(a^m\psi_m)u$, where only the symbol
is subjected to frequency modulation, 
the definition is consistent with the product $au$. 
However, the expression $a^m(x,D)u^m$, that enters \eqref{stabl-id}, 
is symmetric in this sense.

Thirdly, continuity properties of $a(x,D)$ can be
conveniently analysed using Littlewood--Paley techniques applied to both the
symbol $a$ and the distribution $u$. This is facilitated 
because the Fourier multiplication 
by $\psi_m$ occurs in both entries of $a^m(x,D)u^m$. Indeed, one can take
$\psi_m$ to be the first $m+1$ terms in a Littlewood--Paley partition of
unity $1=\sum_{j=0}^\infty \Phi_j$; then bilinearity gives a direct
transition to the paradifferential splitting that has been used repeatedly
for $L_p$ continuity results since the 1980s. The reader is referred to
Section~\ref{cont-sect} for details.

\begin{rem}
  \label{prod-rem}
Analogously to \eqref{aulim-eq}, 
there is an extension of the pointwise
product $(u_1,u_2)\mapsto u_1u_2$, where $u_j\in L_{p_j}^{\loc}(\Rn)$ for
$j=1,2$ with $\tfrac{1}{p_1},\tfrac{1}{p_2},\tfrac{1}{p_1}+\tfrac{1}{p_2}
\in [0,1]$, to the pairs $(u,v)$ in $\cal S'(\Rn)\times\cal S'(\Rn)$ for which
there is a $\psi$-independent limit
\begin{equation}
  \pi(u,v):=\lim_{m\to\infty} u^m v^m
\end{equation}
This general product $\pi(u,v)$ was 
introduced and extensively analysed with paramultiplication in 
\cite{JJ94mlt}; eg 
the convergence in \eqref{aulim'-eq} follows directly from 
\cite[Prop.~3.6]{JJ94mlt}. By \eqref{aulim'-eq} one recovers $\pi(a,u)$ from
\eqref{aulim-eq} when the symbol $a(x,\eta)$ is independent of $\eta$.
(An open question for $\pi(\cdot,\cdot)$ is settled in Theorem~\ref{pi-thm},
where partial associativity is proved from the fact that multiplication by
$C^\infty $-functions commutes with vanishing frequency modulation.)
\end{rem}

\bigskip

The definition sketched in \eqref{stabl-id}
was used rather implicitly in recent works of the
author \cite{JJ04DCR,JJ05DTL}.
In the present article, the purpose is to introduce the
definition of $a(x,D)u$ in \eqref{aulim-eq} systematically and to show that 
it is consistent with \eqref{auint-eq}.

Section~\ref{prep-sect} gives a review of notation and some preparations,
whereas in Section~\ref{special-sect} the special properties of
type $1,1$-operators are elaborated. Section~\ref{basic-sect} deals with
preliminary extensions of type $1,1$-operators,
using cut-off techniques.  The general definition of $a(x,D)$ is
given in Section~\ref{gdef-sect}, where it is proved to be
consistent with the usual one if, say $a(x,\eta)$ coincides (for $\eta$
running through an open set $\Sigma\subset\Rn$)
with an element of the classical symbol
class $S^d_{1,0}$, or  $S^d_{\rho,\delta}$ with
$\rho>\delta$. 
Section~\ref{pslp-sect} contains the proof of the pseudo-local property.
As a preparation, extended action of the bracket $\dual{\cdot}{\cdot}$ from
distribution theory is studied in Section~\ref{exdist-sect}, with
consequences for distribution kernels.
A control of $\supp a(x,D)u$ is proved in Section~\ref{supp-sect}, as
is the spectral support rule in a general version. 
Finally Section~\ref{cont-sect} deals with continuity in the Sobolev spaces
$H^s_p$ and a quick review of the consequences for composite functions.

\section{Notation and Preparations} \label{prep-sect}
The distribution spaces $\cal E'$, $\cal S'$ and $\cal D'$, that are dual to
$C^\infty$, $\cal S$ and $C^\infty_0$ respectively, have the usual meaning
as in eg \cite{H}. $\cal O_M(\Rn)$ stands for the space of slowly increasing
functions, ie the $f\in C^\infty (\Rn)$ such that to every multiindex
$\alpha$ there are $C_\alpha>0$, $N_\alpha>0$ such that $|D^\alpha f(x)|\le
C_\alpha(1+|x|)^{N_\alpha}$ for all $x\in \Rn$.
In addition $C^\infty_{\b}$ denotes the 
Frech\'et  space of smooth
functions with bounded derivatives of any order.
The Sobolev space $H^s_p(\Rn)$ with $s\in \R$ and $1<p<\infty $ is normed by
$\nrm{u}{H^s_p}=\nrm{\cal F^{-1}((1+|\xi|^2)^{s/2}\cal Fu)}{p}$, whereby
$\nrm{u}{p}=(\int_{\Rn}|u|^p\,dx)^{1/p}$ is the norm of $L_p(\Rn)$;
similarly $\nrm{\cdot }{\infty }$ denotes that of $L_\infty (\Rn)$.
That a subset $M$ of $\Rn$ has compact closure is indicated by $M\Subset \Rn$.
As usual $c$ denotes a real constant specific to the place of occurrence.

With the short-hand
$\ang{\xi}=(1+|\xi|^2)^{1/2}$, a symbol $a(x,\eta)$ is said to be in
$S^d_{\rho,\delta}(\Rn\times\Rn)$ if 
$a\in C^\infty(\R^{2n})$
and for all multiindices $\alpha$, $\beta$ there exists
$C_{\alpha,\beta}\ge0$ such that
\begin{equation}
  |D^\alpha_{\xi}D^{\beta}_{x}a(x,\xi)|\le C_{\alpha,\beta}
   \ang{\xi}^{d-\rho|\alpha|+\delta|\beta|}.
  \label{Sdrd-eq}
\end{equation}
Here it is assumed that the order $d\in \R$ and $0<\rho\le1$,
$0\le\delta\le1$ with $\delta\le\rho$,
which is understood throughout unless further restrictions are given.

Along with this there is a pseudo-differential operator $a(x,D)$ defined on
every $u$ in
the Schwartz space $\cal S(\Rn)$ by the Lebesgue integral
\begin{equation}
  a(x,D)u(x)=\OP(a)u(x)= (2\pi)^{-n}\int_{\Rn} e^{\im x\cdot\eta} a(x,\eta)\hat
           u(\eta)\,d\eta.
  \label{axDu-eq}
\end{equation}
Here $\eta$ is the dual variable to $y\in \Rn$ ($u$ is seen as a function of
$y$), while $\xi$ is used for the dual variable to $x$.
If $\psi\in C^\infty_0(\Rn)$ and $\psi=1$ near $0$, then
$\psi_m=\psi(2^{-m}\cdot)$ gives:

\begin{lem}   \label{am-lem}
$a^m(x,\eta)=\psi_m(D_x)a(x,\eta)$
belongs to $S^d_{\rho,\delta}(\Rn\times\Rn)$ when $a$ itself does so, and
$a(x,\eta)=\lim_{m\to\infty}a^m(x,\eta)\psi_m(\eta)$ holds in
$S^{d'}_{\rho,\delta}(\Rn\times\Rn)$ when $d'\ge d+\delta$ and
$d'>d$.
\end{lem}
\begin{proof}
Since $a\in S^d_{\rho,\delta}$  is bounded with respect to $x$, the first
part results from
\begin{equation}
  |D^\beta_xD^\alpha_\eta a^m(x,\eta)|
  \le \int |\check\psi(y)|
           |D^\beta_x D^\alpha_{\eta}a(x-2^{-m}y,\eta)|\,dy
    \le C'_{\alpha,\beta} \ang{\eta}^{d-\rho|\alpha|+\delta|\beta|}.
  \label{am-ineq}
\end{equation}
Since $\psi(0)=1$ the mean value theorem gives $a^m\to a$ in
$S^{d+\delta}_{\rho,\delta}$; and for any $d'>d$ one has
$a^m(x,\eta)\psi_m(\eta)-a^m(x,\eta)\to0$ in $S^{d'}_{\rho,\delta}$; 
whence $a^m(x,\eta)\psi_m(\eta)- a\to 0$. 
\end{proof}

It is straightforward to show from \eqref{axDu-eq} that the bilinear map
\begin{equation}
  \OP\colon S^d_{\rho,\delta}(\Rn\times\Rn)\times\cal S(\Rn)
  \longrightarrow\cal S(\Rn)  
  \label{bilSS-eq}
\end{equation}
is continuous.
Hereby $\cal S(\Rn)$ has a Fr\'echet space structure with seminorms
\begin{equation}
  \norm{\psi}{\cal S,N}=\sup\{\,|\ang{x}^ND^\beta\psi(x)|\mid
   x\in \Rn,\,|\beta|\le N\,\};  
  \label{Snorm-eq}
\end{equation}
whilst $S^d_{1,1}(\Rn\times\Rn)$ is a Fr\'echet space with
the least $C_{\alpha,\beta}$ in
\eqref{Sdrd-eq} as seminorms. 

With $a$ fixed in $S^\infty_{1,1}:=\bigcup_d S^d_{1,1}$ the map
\eqref{bilSS-eq} induces a continuous operator 
$a(x,D)\colon \cal S(\Rn)\to\cal S(\Rn)$, 
that cannot in general be extended to a continuous map 
$\cal S'(\Rn)\to \cal D'(\Rn)$; this is well known cf
Section~\ref{special-sect} below.

The next lemma extends \cite[Lem.~8.1.1]{H} from $u\in \cal E'(\Rn)$
to general $u\in \cal S'(\Rn)$. The extension is irrelevant for the
definition of wavefront sets $\op{WF}(u)$, but useful for calculations.
It is hardly a surprising result, but
without an adequate reference a proof is given here.
Recall that $V\subset\Rn$ is a cone if $\R_+V\subset V$. Throughout
$\R_{\pm}=\{\,t\in \R\mid {\pm}t>0\,\}$. 

First the singular cone $\Sigma(u)$ is defined as the complement in
$\Rn\setminus\{0\}$ of those $\xi\ne0$ contained in an open cone 
$\Gamma\subset\Rn\setminus\{0\}$ fulfilling that 
$\hat u$ is in $L_1^{\loc}$ over $\Gamma$ and
\begin{equation}
  C_N:= \sup_{\Gamma}\ang{\eta}^N|\hat u(\eta)|<\infty,\qquad N>0.
\end{equation}  
Then $\Sigma(u)=\emptyset$ when $u\in \cal S(\Rn)$, and only then (the unit
sphere $\BBb{S}^{n-1}$ is compact).

\begin{lem}
  \label{WF-lem}
Whenever $u\in \cal S'(\Rn)$, then $\Sigma(\varphi u)\subset \Sigma(u)$
for all $\varphi\in C^\infty_0(\Rn)$, and
\begin{equation}
  \op{WF}(u)\subset \singsupp u\times \Sigma(u).
  \label{WFsS-eq}
\end{equation}
\end{lem}

\begin{proof}
It is well known that $\cal S*\cal S'\subset\cal O_M$, so 
$\widehat{\varphi u}(\xi)= (2\pi)^{-n}\dual{\hat u}{\hat\varphi(\xi-\cdot)}$
is $C^\infty$.

Given a cone $\Gamma$ disjoint from $\Sigma(u)$,
it suffices to show that 
$\sup_{\Gamma_1}\ang{\eta}^N|\widehat{\varphi u}(\eta)|<\infty$ on every 
closed cone $\Gamma_1\subset\Gamma\cup\{0\}$ with supremum independent of
$\Gamma_1$.  
When $\xi\ne0$ is fixed in $\Gamma_1$, then $\tfrac{\xi}{|\xi|}\in
\Gamma_1\cap \BBb{S}^{n-1}$ and this set has distance $d>0$ to
$\Rn\setminus\Gamma$, so for $0<\theta<1$ one has $\eta\in \Gamma$ in the
cone $V_{\theta}=\{\,\eta\ne0\mid |\xi-\eta|<\theta d|\xi|\,\}$. 

Supposing $\hat u=0$ in $B(0,\tfrac{1}{4})$, one can take 
$\chi_0+\chi_1=1$ on $\BBb{S}^{n-1}$ such that $\chi_0\in
C^\infty(\BBb{S}^{n-1})$, $\chi_0(\zeta)=1$ for
$|\zeta-\tfrac{\xi}{|\xi|}|<\tfrac{d}{3}$ and $\chi_0(\zeta)=0$ for 
$|\zeta-\tfrac{\xi}{|\xi|}|\ge\tfrac{d}{2}$.
Then $\eta\ne0$ gives 
\begin{equation}
  \hat \varphi(\xi-\eta)=\varphi_0(\eta)+\varphi_1(\eta),
  \quad\text{for}\quad \varphi_j(\eta):=\chi_j(\tfrac{\eta}{|\eta|})
   \hat \varphi(\xi-\eta),
\end{equation}
and both terms are in $C^\infty(\Rn\setminus\{0\})$ with respect to $\eta$,
by stereographic projection and the chain rule.
Now $\supp\varphi_0\subset V_{2/3}\subset\Gamma$ and $\hat u\in
L_1^{\loc}(\Gamma)$ with rapid decay in $\Gamma$ so that one can estimate
$\dual{\hat u}{\varphi_0}$ by means of an integral,
\begin{equation}
  \begin{split}
  \ang{\xi}^N |\dual{\hat u}{\varphi_0}|
  &\le 2^N
   \int_{\Gamma} \ang{\xi-\eta}^N|\chi_0(\tfrac{\eta}{|\eta|})|
  |\hat \varphi(\xi-\eta)|\ang{\eta}^N|\hat u(\eta)|\,d\eta
  \\
  &\le 2^NC_{N+n+1}\nrm{\hat \varphi\ang{\cdot}^N}{\infty}
     \nrm{\chi_0}{\infty} \int \ang{\eta}^{-n-1}\,d\eta<\infty.
  \end{split}
\end{equation}
In $\Rn\setminus V_{1/3}$ one has $|\xi|\le\tfrac{3}{d}|\xi-\eta|$ and
$|\eta|\le (1+\tfrac{3}{d})|\xi-\eta|$, so using the seminorms in
\eqref{Snorm-eq}, 
\begin{equation}
  \begin{split}
  \ang{\xi}^N |\dual{\hat u}{\varphi_1}|
  & \le c\sup_{\eta\notin V_{1/3}, |\alpha|\le M} \ang{\xi}^N \ang{\eta}^M
    |D^\alpha_\eta(\chi_1(\tfrac{\eta}{|\eta|})\hat \varphi(\xi-\eta))|
\\
  & \le c (1+\tfrac{3}{d})^{N+M}
  (\sum_{|\alpha|\le M} \nrm{D^\alpha\chi_1}{\infty})
  \norm{\hat \varphi}{\cal S,N+M}<\infty.
  \end{split}
  \label{V13-ineq}
\end{equation}

Finally one can take
$\tilde \chi\in C^\infty_0(\Rn)$ such that $\tilde \chi=1$ for $|\eta|\le
\tfrac{1}{4}$ with support in $B(0,\tfrac{1}{2})$, then 
$\ang{\xi}^N|\dual{\hat u}{\tilde \chi}|\le c
\norm{\tilde \chi}{\cal S,N+M}\norm{\hat \varphi}{\cal S, N+M}$
follows as in \eqref{V13-ineq}. 

All bounds are uniform in $\xi$ and in $\Gamma_1$, hence 
$\sup_{\Gamma}\ang{\cdot}^N|\widehat{\varphi u}|<\infty$.
This proves $\Sigma(\varphi u)\subset \Sigma(u)$, 
so \eqref{WFsS-eq} holds as 
$\op{WF}(u)=\{\,(x,\xi)\mid \xi\in
\bigcap_{\varphi(x)\ne0,\ \varphi\in C_0^\infty} \Sigma(\varphi u)\,\}.
$
\end{proof}

\begin{rem}
  \label{WF-rem}
In Lemma~\ref{WF-lem} equality obviously holds in \eqref{WFsS-eq} if the
singular cone is a ray, ie if $\Sigma(u)=\R_+\zeta$ for some $\zeta\in \Rn$.
In such cases $\op{WF}(u)$ can be easily determined.
\end{rem}

\section{Special properties of type $\mathbf{1},\mathbf{1}$-operators}   \label{special-sect}
Many of the pathological properties of type 1,1-operators can be obtained
from simple examples of the form
\begin{equation}
  a_{\theta}(x,\eta)=\sum_{j=1}^\infty 2^{jd}e^{-\im 2^j{x}\cdot{\theta}}
             \chi(2^{-j}\eta),
  \label{ching-eq}
\end{equation}
whereby $\chi\in C^\infty_0(\Rn)$ with $\supp\chi\subset\{\,\eta\mid
\tfrac{3}{4}\le |\eta|\le\tfrac{5}{4}\,\}$; and $\theta\in \Rn$ is fixed.
Clearly $a_{\theta}$ is in $S^d_{1,1}$ since the terms are disjointly supported.

Such symbols were used by C.~H.~Ching \cite{Chi72} and 
G.~Bourdaud \cite{Bou88} for
$d=0$, $|\theta|=1$ to show $L_2$-unboundedness.
Refining their counter-examples, L.~H{\"o}rmander 
linked continuity from $H^s$ with $s\ge-r$, $r\in \N_0$,
to the property that $\theta$ is a zero of $\chi$ of order $r$.

This is generalised to $\theta\in \Rn$ here because \eqref{ching-eq} 
with $|\theta|\ne1$ enters the
proof that type $1,1$-operators do not always preserve wavefront sets.
And by consideration of arbitrary orders $d\in \R$ the
counter-examples get interesting additional properties; cf
Remark~\ref{Weierstrass-rem} ff.

From the definition of $a_{\theta}(x,\eta)$ in 
\eqref{ching-eq} it is clear that
$\hat u\in C^\infty_0(\Rn)$ gives
\begin{equation}
  a_{\theta}(x,D)u= \sum_{j=1}^\infty (2\pi)^{-n} \int e^{\im x\cdot\eta}
             2^{jd}\chi(2^{-j}\eta+\theta)\hat u(\eta+2^j\theta)\,d\eta.
\end{equation}
Then the adjoint $b_{\theta}(x,D)\colon \cal S'(\Rn)\to\cal S'(\Rn)$ of $a_{\theta}(x,D)$
fulfils, for all $v\in \cal S(\Rn)$,
\begin{equation}
  \dual{u}{\overline{b_{\theta}(x,D)v}}=\dual{a_{\theta}(x,D)u}{\overline{v}}
  = (2\pi)^{-n}\int \hat u(\xi)
      (\sum_{j=1}^\infty 2^{jd}\bar\chi(2^{-j}\xi)
       \hat v(\xi-2^j\theta))^{\overline{\ }}\,d\xi,
\end{equation}
so
$\cal F b_{\theta}(x,D)v(\xi)= \sum_{j=1}^\infty 2^{jd}\bar\chi(2^{-j}\xi)
\hat v(\xi-2^j\theta)$.
This gives a convenient way to calculate the $H^s$-norm of $b_{\theta}(x,D)v$, for
when this is finite the disjoint supports of the $\chi(2^{-j}\cdot)$ imply
\begin{equation}
  \begin{split}
  (2\pi)^n\nrm{b_{\theta}(x,D)v}{H^s}^2
&=\sum_{j=1}^\infty\int_{2^{j-1}<|\xi|<2^{j+1}}
  (1+|\xi|^2)^s 4^{jd}|\chi(2^{-j}\xi)\hat v(\xi-2^j\theta)|^2\,d\xi
\\
&=\int \Big( \sum_{\tfrac{1}{2}<|2^{-j}\xi+\theta|<2}
              (1+|\xi+2^j\theta|^2)^s 4^{jd}|\chi(2^{-j}\xi+\theta)|^2
       \Big)|\hat v(\xi)|^2\,d\xi.
  \end{split}
  \label{bHs-eq}
\end{equation}
Therefore the action of $b_{\theta}(x,D)$ `piles up' at 
$\xi=0$, say for $s>0=d$, $|\theta|=1$; ie the adjoint $b_{\theta}(x,D)$ is not even
of type $1,1$ (cf \cite{Bou88}). 
Unless of course $\chi(\theta)=0$.

This leads to the next result, which for $d=0$, $|\theta|=1$ gives back 
\cite[Prop.~3.5]{H88}. The identity \eqref{bHs-eq} is taken from the proof
given there, but \eqref{bHs-eq} is used consistently here.

\begin{prop}   \label{Hoer-prop}
When $a_{\theta}(x,\eta)$ is given by \eqref{ching-eq} for $d\in \R$, $\theta\ne0$,
then $a_{\theta}(x,D)$ extends by continuity to a bounded operator 
$H^{s+d}(\Rn)\to H^s(\Rn)$ for all
\begin{equation}
  s>-r,\quad r:=\inf\{\,|\alpha|\mid D^\alpha\chi(\theta)\ne0\,\}.
  \label{Dachi-eq}
\end{equation}
Conversely, for $|\theta|\in [\tfrac{3}{4},\tfrac{5}{4}]$
existence of a continuous linear extension 
$H^{s+d}(\Rn)\to\cal D'(\Rn)$ implies \eqref{Dachi-eq}.
\end{prop}
\begin{proof}
For sufficiency
of \eqref{Dachi-eq} it is enough to prove the adjoint $b_{\theta}(x,D)$ continuous 
\begin{equation}
  b_{\theta}(x,D)\colon H^t(\Rn)\to H^{t-d}(\Rn)\quad\text{for}\quad t<r.
\end{equation}
This is obtained from \eqref{bHs-eq} with $s=t-d$, where the inequalities 
$\tfrac{1}{2}<|2^{-j}\xi+\theta|<2$ yield that 
$(1+4^{jd}|2^{-j}\xi+\theta|^2)^{t-d}=\cal O(4^{j(t-d)})$ for $j\to\infty$.
Moreover $|\chi(\theta+\eta)|\le c|\eta|^r$, so 
\begin{equation}
  4^{jd}(1+|\xi+2^j\theta|^2)^{t-d}|\chi(2^{-j}\xi+\theta)|^2
  \le c|\xi|^{2r}4^{j(t-r)}.
\end{equation}
Since $t-r<0$ the geometric series can be estimated by the first term, and
using that
$|2^{-j}\xi|-|\theta|\le2$ implies $2^j+1\ge \tfrac{|\xi|+1}{2+|\theta|}$,
and hence $2^j\ge\tfrac{\ang{\xi}}{4+2|\theta|}$, this gives
\begin{equation}
    \nrm{b_{\theta}(x,D)v}{H^{t-d}}^2\le \frac{c}{1-4^{t-r}}
  \int |\xi|^{2r}\bigl(\frac{\ang{\xi}}{4+2|\theta|}\bigr)^{2(t-r)}
  |\hat v(\xi)|^2\,d\xi \le c\nrm{v}{H^t}^2.
\end{equation}

Necessity of \eqref{Dachi-eq} 
for $\tfrac{3}{4}\le|\theta|\le\tfrac{5}{4}$ follows if
$a_{\theta}(x,D)$ cannot be continuous $H^{-r+d}\to\cal D'$.
So the adjoint is assumed continuous $b_{\theta}(x,D)\colon
C^\infty_0(\Rn)\to H^{r-d}(\Rn)$.
Inserting in \eqref{bHs-eq} that
$\chi(\theta+\eta)=P_r(\eta)+\cal O(|\eta|^{r+1})$
for a homogeneous polynomial $P_r\not\equiv0$ of degree $r$,
\begin{equation}
  \nrm{b_{\theta}(x,D)v}{H^{r-d}}^2=  c\int
  (\sum_{\tfrac{1}{2}<|2^{-j}\xi+\theta|<2}(2^{-j}+|2^{-j}\xi+\theta)|^2)^{r-d}
  |P_r(\xi)+\cal O(2^{-j})|^2)|\hat v(\xi)|^2\,d\xi.
\end{equation}
For each $\xi$ the sum runs over all $j\ge J$ for a certain $J\ge0$, since
$\tfrac{1}{2}<|\theta|<2$. By increasing $J$ if necessary, 
$|P_r(\eta)+\cal O(2^{-j})|\ge \tfrac{1}{2}|P_r(\eta)|$. In addition it
holds for all $j\ge0$ that
$(2^{-j}+|2^{-j}\xi+\theta)|^2)^{r-d}\ge
\min(5^{r-d},4^{d-r})$.
Therefore the series above is estimated from below by
$\sum_{j\ge J}\tfrac{1}{4}|P_r(\xi)|^2=\infty$ for $\xi\in \supp\hat v$.
This contradicts that $b_{\theta}(x,D)(C^\infty_0)\subset H^{r-d}$.
\end{proof}

It is with good reason that necessity of \eqref{Dachi-eq} is obtained
only for $\tfrac{3}{4}\le|\theta|\le\tfrac{5}{4}$. For if 
$\theta\notin\supp\chi$, \eqref{Dachi-eq} would hold with
$r=\infty$ and 
$a_{\theta}(x,D)$ be continuous from $H^s$ for every $s\in \R$ by the
sufficiency
(ie no necessary condition can be imposed if 
$|\theta|\notin[\tfrac{3}{4},\tfrac{5}{4}]$).

\subsection{Unclosed graphs}   \label{Ugrph-ssect}
As an addendum to  Proposition~\ref{Hoer-prop}, 
it is a strengthening fact that
Ching's operator $a_{\theta}(x,D)$ can be taken 
unclosable in $\cal S'(\Rn)$. Ie its graph $G$, as
a subspace of $\cal S'(\Rn)\times\cal S'(\Rn)$, can have 
a closure $\overline{G}$ that is not a graph, for as
shown in Lemma~\ref{cex-lem} below  $\overline{G}$ will in some cases contain 
a pair $(0,v)$ for some $v\ne0$, $v\in \cal S(\Rn)$.

This is important since it shows that type $1,1$-operators cannot
just be 
defined by closing their graphs; nor can one hope to give a
definition by other means, such as \eqref{stabl-id}, 
and reach a closed operator in general.

\begin{lem}
  \label{cex-lem}
Let $a_{\theta}(x,\eta)$
be given as in \eqref{ching-eq} for $d\in \R$ and with $|\theta|=1$ 
and $\chi=1$ on the ball $B(\theta,\tfrac{1}{10})$.
Then $a_{\theta}(x,D)$ is unclosable since there exist
$v$, $v_N\in \cal S(\Rn)\setminus\{0\}$ such that
\begin{equation}
    \lim_{N\to\infty}v_N= 0 \text{ in $H^d(\Rn)$},
\qquad
  \lim_{N\to\infty}  a_{\theta}(x,D)v_N = v \text{ in $\cal S(\Rn)$}.
\end{equation}
\end{lem}

\begin{proof}
Take $v\in\cal S(\Rn)\setminus\{0\}$ with
$\supp\hat v\subset \{\,|\xi|\le\tfrac{1}{20}\,\}$ and let
\begin{equation}
  \hat v_N(\xi)= \tfrac{1}{\log N}\sum_{j=N}^{N^2} \tfrac{2^{-jd}}{j}
  \hat v(\xi-2^j\theta)=
  \cal F\bigl(v(x)\sum_{j=N}^{N^2} 
  \frac{2^{-jd}e^{\im 2^jx\cdot\theta}}{j\log N} \bigr).
\end{equation}
Since the $\supp\hat v(\xi-2^j\theta)$ are disjoint, and 
$c_12^j\le \ang{\xi}\le c_2 2^{j}$ hold on each support,
\begin{equation}
    \nrm{v_N}{H^d} =
  \sum_{j=N}^{N^2} \tfrac{1}{(2\pi)^{n/2}\log N}
  (\int_{\Rn}\frac{\ang{\xi}^{2d}}{j^22^{2jd}}
      |\hat v(\xi-2^{j}\theta)|^2\,d\xi)^{1/2}
  \le c\nrm{v}{2}(\sum_N^\infty j^{-2})^{1/2}
  \searrow 0.
  \label{vNnrm-eq}
\end{equation}
Because $v_N$ is defined by a finite sum, 
and $\chi(2^{-j}\cdot)\equiv1$ on $\supp\hat v(\cdot-2^j\theta)$,
a direct computation gives the following limit in $\cal S(\Rn)$,
\begin{equation}
  a_{\theta}(x,D)v_N(x)=\tfrac{1}{\log N}
  (\tfrac{1}{N}+\tfrac{1}{N+1}+\dots+\tfrac{1}{N^2})v(x) \xrightarrow[N\to\infty]{~} v(x).
  \label{atN-eq}
\end{equation}
Indeed,
$1\le (N^{-1}+\dots+N^{-2})/\log N\le {\log(\tfrac{N^2}{N-1})}/{\log N}
\searrow 1$.
\end{proof}

The sequence $v_N$ also tends to $0$ in the more general Besov and
Lizorkin--Triebel spaces $B^{d}_{p,q}$ and $F^{d}_{p,q}$ for every $p\in
[1,\infty]$ and $q>1$; cf \cite{JJ05DTL}.

\subsection{Violation of the microlocal property}   \label{WF-ssect}
In the proof of Lemma~\ref{cex-lem} the role of the exponential
functions in $a_{\theta}(x,\eta)$ was clearly 
to move all high frequencies in the spectrum of $v_N$ to a
neighbourhood of the origin. 
So it is perhaps not surprising that another variant of Ching's example 
will produce frequencies $\eta$ that are moved to, say $-\eta$. 

This indicates that type $1,1$-operators need not have
the microlocal property;
ie the inclusion $\op{WF}(a(x,D)u)\subset\op{WF}(u)$ among
wavefront sets is violated for certain symbols $a\in S^\infty_{1,1}$. 

This is explicated here, following C.~Parenti and L.~Rodino \cite{PaRo78}
who treated $d=0$ and $n=1$. Their suggested programme is carried out below
with a coverage of all $d\in \R$, $n\in \N$ and arbitrary directions of
$\theta$. 
As a minor improvement, the wavefront sets are explicitly 
determined here; and due to the
uniformly estimated symbols and the fact that 
$v$ in \eqref{wtheta-eq} below has
compact spectrum, the present proofs are also rather cleaner.

With notation as in the proof of Lemma~\ref{cex-lem}, again with $|\theta|=1$, 
one can introduce
$w_{\theta}(x)=w(\theta,d;x)= 
\sum_{j=1}^\infty 2^{-jd}e^{\im 2^j\theta\cdot x}v(x)$
for $v\in \cal S(\Rn)$ with $\supp\hat v\subset B(0,1/20)$, 
so that
\begin{equation}
  \hat w_{\theta}(\eta)
  = \sum_{j=1}^\infty 2^{-jd} \hat v(\eta-2^j\theta).
  \label{wtheta-eq}
\end{equation}
As shown below, this distribution has the cone $\Rn\times
(\R_+\theta)$ as its wavefront set.
The counter-example arises by considering $w_{\theta}$ together with the symbol
$a_{2\theta}\in S^d_{1,1}(\Rn\times\Rn)$ 
defined by \eqref{ching-eq} for a $\chi$ fulfilling
\begin{equation}
  \chi(\eta)=1 \quad\text{for}\quad
  \tfrac{9}{10}\le |\eta|\le \tfrac{11}{10}.   
  \label{chi1-eq}
\end{equation}
As $\chi$ vanishes around $2\theta$, there are
by Proposition~\ref{Hoer-prop} continuous extensions
\begin{equation}
  a_{2\theta}(x,D)\colon 
  H^{s+d}(\Rn)\to H^s(\Rn) \quad\text{for all}\quad s\in \R.
  \label{aHs-eq}
\end{equation}
Moreover, it is easy to see that in this case every $(\xi,\eta)$ in 
$\supp \hat a_{2\theta}$ lies in the cone $|\eta|\le 2|\xi+\eta|$ so that
$a$ fulfils \eqref{TDC-cnd} for $C=2$.
So neither a large domain, like $\bigcup H^s$, nor 
the twisted diagonal condition
can ensure the microlocal property of a type $1,1$-operator:

\begin{prop}   \label{WF-prop}
The distributions $w(\theta,d;x)$ are in $H^s(\Rn)$ precisely for $s<d$, and 
when $a_{2\theta}$ is chosen as in \eqref{ching-eq},\eqref{chi1-eq} 
with $|\theta|=1$, then 
\begin{equation}
  a_{2\theta}(x,D)w(\theta,d;x)=w(-\theta,0;x).   
\end{equation}
Moreover,
\begin{align}
  \op{WF}(w_{\theta})&= \Rn\times (\R_+\theta),
  \label{WF-eq}
\\
   \op{WF}(a_{2\theta}(x,D)w(\theta,d;x))&=\Rn\times(\R_+(-\theta)),
\end{align}
so the wavefront sets of $w_{\theta}$ and $a_{2\theta}(x,D)w_{\theta}$ are
disjoint. 
\end{prop}

\begin{proof}
Estimates analogous to \eqref{vNnrm-eq} show that the series for 
$\hat w_{\theta}$ converges in $L_2(\ang{\eta}^{2s}d\eta)$ if and only if
$s<d$; hence $w_{\theta}$ is well defined in $\cal S'$ and belongs to
$H^s$ for $s<d$. 
Since the series for $w_{\theta}$ converges in $H^s$ for $s<d$, the
continuity \eqref{aHs-eq} and \eqref{chi1-eq} imply
\begin{equation}
  a_{2\theta}(x,D)w(\theta,d;x)= \sum_{j=1}^\infty 
   (2\pi)^{-n}\int e^{\im x\cdot(\eta-2^j2\theta)}
  \hat v(\eta-2^j\theta)\,d\eta
  =w(-\theta,0;x).
  \label{axDw-eq}
\end{equation}
Therefore  
$\op{WF}(a_{2\theta}(x,D)w_{\theta})\bigcap \op{WF}(w_{\theta})=\emptyset$ 
follows as soon as \eqref{WF-eq} has been proved.

Clearly $\supp \hat w_{\theta}$ intersects each ray $\R_+\zeta$ only in a
compact set, except for $\zeta=\theta$ in which case 
$|\hat w_{\theta}(\eta)|\le 2^{|d|}\nrm{\hat v}{\infty}\ang{\eta}^{-d}$ 
is an exact decay rate as
$2^{j-1}\le \ang{\eta}\le 2^{j+1}$ on $\supp\hat v(\eta-2^j\theta)$, so
that $\Sigma(w_{\theta})=\R_+\theta$ in the notation of Lemma~\ref{WF-lem}. 
This almost proves \eqref{WF-eq}, but a full proof is a little lengthy, because
of the overlapping supports in 
\begin{equation}
  \widehat{\varphi w_{\theta}}=\sum 2^{-jd}
  \widehat{\varphi v}(\eta-2^j\theta),
  \qquad \varphi\in C^\infty_0(\Rn). 
\end{equation}
(This important technicality seems to be overlooked in the sketchy
arguments of G.~Garello \cite{Gar94}, 
who also dealt with extensions of the results of \cite{PaRo78}.) 

That $\singsupp w_{\theta}=\Rn$ 
follows if $\varphi\in C^\infty_0(\Rn)\setminus\{0\}$ implies
$\varphi w_{\theta}\notin C^\infty_0(\Rn)$. 
The last property is invariant under multiplication by a character, so it
can be arranged that $|\widehat{\varphi v}|$ attains its maximum at $0$. 
Despite the overlapping supports, 
$\widehat{\varphi w_{\theta}}$ can then be seen to
decay as $\ang{\eta}^{-d}$ along $\R_+\theta$, but not rapidly
because $\widehat{\varphi v}\not\equiv0$.

To carry this out, one can pick $r\in \,]0,\tfrac{1}{4}[\,$ so that
\begin{equation}
  |\eta|<r \implies |\widehat{\varphi v}(\eta)|> 
  \tfrac{1}{2}\nrm{\widehat{\varphi v}}{\infty}.
  \label{rest-eq}
\end{equation}
Since every term in $\widehat{\varphi w_{\theta}}$ is in $\cal S(\Rn)$ it is
only necessary to estimate those with indices $j>J$, for some $J$.
(The estimates make sense since $\widehat{\varphi w_{\theta}}$ 
is a function, in $L_2(\ang{\eta}^{2s}d\eta)$.)
Using that
$c_N:=\sup|\eta|^N|\widehat{\varphi v}(\eta)|<\infty$, and $r<1/4$,
one finds with a fixed $N>|d|$ that for $\eta\in B(2^k\theta,r)$, $k>J$, 
\begin{equation}
  2^{kd}\sum_{J<j<k} 2^{-jd}|\widehat{\varphi v}(\eta-2^j\theta)|
  \le \sum_{J<j<k} 2^{(k-j)d}\frac{c_N}{2^{kN}} (1-\frac{2^{j}}{2^k}
   -\frac{r}{2^k})^{-N}
  \le \frac{c_N4^N}{ 2^{JN}(1-2^{d-N})}.
\end{equation}
Similarly
$\sum_{j>k} 2^{(k-j)d}|\widehat{\varphi v}(\eta-2^j\theta)|
\le \frac{c_N4^N}{ 2^{JN}(1-2^{-d-N})}$
results by factorising $2^{j}$ out of $(\dots)^{-N}$.

It is clear one can take $J$ so large that the right-hand sides are less
than $\tfrac{1}{5}\nrm{\widehat{\varphi v}}{\infty}$. Then \eqref{rest-eq}
and the fact that $2^{k-1}\le\ang{\eta}\le 2^{k+1}$ on $B(2^k\theta,r)$
give
\begin{equation}
  |\sum_{j>J} 2^{-jd}\widehat{\varphi v}(\eta-2^j\theta)| 
  \ge  (\tfrac{1}{2}-\tfrac{2}{5})
   \nrm{\widehat{\varphi v}}{\infty} 
   2^{|d|}\ang{\eta}^{-d} \quad\text{for}\quad \eta\in B(2^k\theta,r).
  \label{phiv-ineq}
\end{equation}
This estimate is uniform in $k>J$;
hence $\singsupp w_{\theta}=\Rn$.
So by Lemma~\ref{WF-lem} and Remark~\ref{WF-rem}, 
$\op{WF}(u)=\singsupp w_{\theta}\times \Sigma(w_{\theta})=\Rn\times 
(\R_+\theta)$, ie \eqref{WF-eq} is obtained. 
\end{proof}

\begin{rem}
It is clear from \eqref{axDw-eq} that
$a_{2\theta'}(x,\eta)w(\theta,d;x)=
w(\theta'',0;x)$ for $\theta''=\theta+2\theta'$, $|\theta'|=1=|\theta|$. But
$\theta''$ can point in any direction in $\Rn$, so 
type $1,1$-operators can make arbitrary directional changes in wavefront
sets (as noted in \cite{PaRo78}).
\end{rem}

\begin{rem}   \label{Weierstrass-rem}
There is an amusing reason why the counter-example $w_{\theta}$ in
Proposition~\ref{WF-prop} is singular on all of $\Rn$, ie why $\singsupp
w_{\theta}=\Rn$. In fact $w_{\theta}(x)= v(x)f(x\cdot \theta)$ where 
$f(t)=\sum_{j=1}^\infty 2^{-jd}e^{\im 2^jt}$, and this is for $0<d\le1$ 
a well-known  variant of Weierstrass's  
nowhere differentiable function (a fact that could have substantiated the
argument for formula \mbox{(19)} in \cite{PaRo78}).
That the theory of type $1,1$-operators is linked to this
classical construction seems to be previously unobserved.
\end{rem}

\begin{rem}   \label{f-rem}
To elucidate Remark~\ref{Weierstrass-rem},
$f(t)=\sum_{j=1}^\infty 2^{-jd} e^{\im 2^jt}$ 
is investigated here.
Clearly $f\in \cal S'(\R)$ for all $d\in \R$, as the Fourier transformed
series $2\pi\sum_{j=1}^\infty 2^{-jd}\delta_{2^j}$ converges there.

By uniform convergence $f$ is for $d>0$
a continuous $2\pi$-periodic and bounded function.
Nowhere-differentiability for $0<d\le 1$ is an easy (maybe not widely known)
exercise in distribution theory: 
$\supp\cal F f$ is lacunary, so any
choice of $\chi\in \cal S(\R)$ such that $\hat \chi(1)=1$ and $\supp\hat
\chi\subset \,]\tfrac{1}{2},2[\,$ will give
$\supp \hat \chi(2^{-k})\bigcap \supp\delta_{2^j}\ne\emptyset$ only for 
$j= k$, which entails
\begin{equation}
  2^{-kd}e^{\im 2^kt}= 2^k\chi(2^k\cdot)*f(t)
  =\int_{\R} \chi(z)(f(t-2^{-k}z)-f(t))\,dz;
  \label{fint-eq}
\end{equation}
so if $f$ were differentiable at $t_0$, 
$G(h):=\tfrac{1}{h}(f(t_0+h)-f(t_0))$ would be in $C(\R)\cap L_\infty (\R)$,
and the contradiction $d>1$ 
would follow since by majorised convergence 
\begin{equation}
  2^{(1-d)k}e^{\im 2^kt_0}
  = -\int G(-2^{-k}z)z \chi(z)\,dz
  \xrightarrow[k\to\infty]{~} f'(t_0)\cdot(-D)\hat \chi(0)=0.
  \label{Gint-eq}
\end{equation}
Moreover, if $m<d\le m+1$ for $m\in \N$ it follows by termwise differentiation
that $f\in C^{m}(\R)$, but with $f^{(m)}$ nowhere differentiable; so
$\singsupp f=\R$ for all $d>0$.  

For $d\le0$ one has $f\notin L_1^{\loc}$, for $f$ is invariant
in $\cal S'$ under translation by $2\pi$, so if $f\in
L_1^{\loc}$ is assumed, $\dual{f}{\varphi}=\int f\varphi\,dt$ holds for 
$\varphi$ in $C^\infty_0(\R)$ as well as in $\cal S$, since 
$|\int f\varphi\,dt|\le c \sup_{t\in \R}(1+|t|^2)|\varphi(t)|$ 
follows from the fact that $(1+r^2)^{-1}f(r)$ is in $L_1$:
\begin{equation}
  \int_\R \frac{|f(r)|}{1+r^2}\,dr =\sum_{p\in \Z} \int_{2p\pi}^{2(p+1)\pi}
    \frac{|f(r)|}{1+r^2}\,dr\le \int_0^{2\pi}|f|\,dr 
   \sum_{p=0}^\infty \tfrac{2}{1+(2p\pi)^2}<\infty .
  \label{fr2-eq}
\end{equation}
Therefore the convolution in \eqref{fint-eq} is given by the integral also
in this case. By taking $t$ outside a $G_{\delta}$-set $G$ of measure
$0$, $f$ is continuous (in $\R\setminus G$) at $t$, whence
\begin{equation}
  2^{-kd}=|\int 2^k\chi(2^kr)(f(t-r)-f(t))\,dr|\xrightarrow[k\to\infty ]{~}0.
\end{equation}
In fact, for $\varepsilon>0$ the part with $|r|<\delta$ is
$<\varepsilon$ for some $\delta>0$, but $\sup_{|r|\ge
\delta}(1+|t-r|^2)2^k|\chi(2^kr)|=\cal O(2^{-k})$, so since $L_1(\R)$
by \eqref{fr2-eq}
contains $r\mapsto (f(t-r)-f(t))/(1+|t-r|^2)$ the limit $0$ results. Thence
the contradiction $d>0$.

To complete the picture, Weierstrass' original function 
$W(t)=\sum_{j=0}^\infty a^{-j}\cos(b^jt)$, where $b\ge a>1$, 
is nowhere differentiable by the same
argument. One only has to take $\supp\cal F\chi\subset
\,]\tfrac{1}{b},b[\,$, $\cal F\chi(1)=1$,
for in $\cal F\cos(b^j\cdot )=\tfrac{2\pi}{2}(\delta_{b^j}+\delta_{-b^j})$ the
last term is removed by $\hat \chi(b^{-j}\cdot )$, so that $\chi(b^{-j}D)$
yields a second microlocalisation of $W$.
As in \eqref{fint-eq}--\eqref{Gint-eq} it follows that $(\tfrac{b}{a})^k
e^{\im b^kt_0}\to0$ for $k\to\infty $, contradicting that $b\ge a$.
A further study of nowhere differentiable functions by means of 
microlocalisation can be found in \cite{JJ10now}.
\end{rem}

\begin{rem}  \label{freg-rem}
As a precise account of the regularity,
$f\in B^d_{\infty ,\infty }(\R)$; for $0<d<1$ this Besov space consists of
H{\"o}lder continuous functions of order $d$.
Indeed, the norm of $f$ in $B^d_{\infty ,\infty }$ 
is from the left part of \eqref{fint-eq} seen to equal $1$,
 when $\chi$ is taken
as $\Phi_1$ in a suitable Littlewood--Paley partition of unity 
$1=\sum_{j=0}^\infty \Phi_j$.
Moreover, $f\in F^d_{p,\infty;\loc}(\R)$ when
$d\in \R$, $0<p<\infty$, for
the definition in \cite{T2} of $F^s_{p,q}$ gives, 
when $v\in \cal S(\R)$ with $\supp\hat v\subset B(0,\tfrac{1}{20})$, 
\begin{equation}
  \nrm{vf}{F^d_{p,\infty}}=
  (\int (\sup_j 2^{jd}|\Phi_j(D)(vf)|)^p\,dt)^{1/p}
  =\Nrm{\sup_j |v(t)e^{\im 2^jt}|}{p}
  =\nrm{v}{p}.
\end{equation}
These are identities, so the $F^d_{p,\infty}$-regularity is sharp.
That $\varphi f\in F^s_{p,\infty}(\R)$ for $\varphi\in C^\infty_0(\R)$
results from $\varphi f=\tfrac{\varphi}{v} vf\in F^d_{p,\infty}$ 
when $\varphi\in C^\infty_0(\R)$
has support in $\R\setminus\{v=0\}$; one can reduce to this with
a partition of unity on $\varphi$ and translation of $v$ in each term.
\end{rem}
\begin{rem}  \label{wreg-rem}
To substantiate Remark~\ref{Weierstrass-rem}, note that 
$w_{\theta}(x)= v(x)f(x\cdot \theta)$ is almost nowhere differentiable for
$0<d\le1$, since $v$ has isolated zeroes.
If $d\le0$ then $w(\theta,d;\cdot)\notin
L_1^{\loc}(\Rn)$ for else one can derive the contradiction 
$2^{-kd}v(x)=2^{kn}\chi(2^k\cdot)*[vf(\dual{\cdot }{\theta})]\to0$ 
by modifying the
corresponding part of Remark~\ref{f-rem}.
As in Remark~\ref{freg-rem} it follows that
\begin{equation}
  w(\theta,d;x)\in F^d_{p,\infty;\loc}(\Rn)
  \quad\text{for all $p\in \,]0,\infty[\,$, $d\in \R$}.
\end{equation}
If $\tilde w(\theta,d;x)$ is defined as $w(\theta,d;x)$ except with a
further factor $1/j$ in each summand, similar arguments yield that
$\tilde w\in H^s$ for $s\le d$ as well as the other properties in
Proposition~\ref{WF-prop}, with a nowhere differentiable series for
$0<d<1$. Moreover, $\tilde w(\theta,d;\cdot )$ is in $F^d_{p,q;\loc}(\Rn)$
as soon as $q>1$ for every $p\in \,]0,\infty [\,$.
Hence the counter-examples with unclosable graphs and violated microlocal
properties are both related to Lizorkin--Triebel spaces $F^d_{p,q}$ with
arbitrary $q>1$; cf \eqref{Fdp1-eq} and Section~\ref{Ugrph-ssect}.
\end{rem}

\section{Preliminary extensions}   \label{basic-sect}
Throughout $\cal F_{y\to\eta}$ etc.\ will denote partial 
Fourier transformations, that are all homeomorphisms on
$\cal S'(\R^{2n})$. 
In general the Fourier transformation in all variables is written $\cal Fu$
or $\hat u$, except that for a symbol $a\in \cal S'(\R^{2n})$,
\begin{equation}
  \hat a(\xi,\eta):=\cal F_{x\to\xi}a(\xi,\eta).
\end{equation} 
Transformation of coordinates via $(x,y)\mapsto(x,x-y)$, that has matrix 
$M=\left(\begin{smallmatrix} I&0\\I&-I\end{smallmatrix}\right)=M^{-1}$, 
is indicated by $f\circ M$.

As a preparation some well-known formulae are recalled:

\begin{prop}   \label{aK-prop}
 Let $a\in
S^\infty_{1,1}(\Rn\times\Rn)$ and $u,v,f,g\in \cal S(\Rn)$. Then
\begin{gather}
  \dual{a(x,D)u}{v}= 
  \dual{\tfrac{e^{\im x\cdot\eta}}{(2\pi)^{n}}a(x,\eta)}{v(x)\otimes\hat u(\eta)}
   =\dual{K}{v\otimes u}.
\label{aK-eq}
\\
  K(x,y)=\cal F^{-1}_{\eta\to y}( a)(x,x-y) 
        =\cal F^{-1}_{\eta\to y}( a(x,\eta)) \circ 
    \begin{pmatrix} I& 0 \\ I & -I\end{pmatrix}.
  \label{Kxy-eq}
\\
  \dual{\cal F a(x,D)f}{\cal F g}=
  (2\pi)^{-n}\iint \hat a(\xi-\eta,\eta)\hat f(\eta)\hat g(\xi)\,d\xi d\eta.
  \label{Fa-eq}
\end{gather}
Here $\iint\dots d\xi d\eta$ is valid as an integral for $a\in \cal
S(\R^{2n})$, but should be read as the scalar product on $\cal
S'\times\cal S$ for general $a\in S^\infty_{1,1}$.
\end{prop}

\begin{proof}
By Fubini's theorem, \eqref{aK-eq} holds for $a\in \cal S(\R^{2n})$, 
when $K$ is given as in \eqref{Kxy-eq}. The bijection $a \leftrightarrow K$
extends to a homeomorphism on $\cal S'(\R^{2n})$. So by density of $\cal S$ in
$S^d_{1,1}$, as subsets of $S^{d+1}_{1,1}$ hence of $\cal S'$, the
identities in \eqref{aK-eq} hold for all $a\in S^d_{1,1}$.
Formula \eqref{Fa-eq} results from \eqref{aK-eq} for
$u=f$, $v=\cal F^2g$, since
$(\cal F^2g)\otimes \hat f= \cal F_{\xi\to x}(\hat g\otimes\hat f)$.
\end{proof}

The partially Fourier transformed symbol $\hat a(\xi,\eta)$ 
is closely related to the distribution kernel $K(x,y)$ of $a(x,D)$ as well
as to the kernel $\cal K(\xi,\eta)$ of the conjugation 
$\cal Fa(x,D)\cal F^{-1}$
of $a(x,D)$ by the Fourier transformation on $\Rn$.
Indeed, modulo simple isomorphisms, $\hat a$ gives both 
$\cal K$ and
the frequencies in $K$:

\begin{prop}   \label{FKah-prop}
When $a\in S^\infty_{1,1}(\Rn\times\Rn)$, and $K$, $\cal K$ and $M$ 
are as above,
\begin{equation}
  \cal F K(\xi,\eta) = \hat a(\xi+\eta,-\eta)=\hat a\circ M^t(\xi,\eta)
  =(2\pi)^n\cal K(\xi,-\eta).
  \label{FKah-eq}
\end{equation}
\end{prop}
\begin{proof}
\eqref{Kxy-eq} implies that 
$K=\cal F^{-1}_{\eta\to y}(e^{-\im x\cdot\eta}a(x,-\eta))$, since $\cal F^{-1}$
commutes with reflections in $\eta$ and $y$. Then \eqref{FKah-eq}
follows by application of $\cal F$ and \eqref{Fa-eq}.
\end{proof}

\bigskip

The right-hand side of \eqref{aK-eq} is inconvenient 
for the definition of type $1,1$-operators, as in general both
entries of $\dual{K}{v\otimes u}$ have singularities
(in some cases this can be handled, cf Section~\ref{exdist-sect}).
However, it is a well-known fact that also in case $\rho=1=\delta$ the
kernels only have singularities along the diagonal.

\begin{lem}   \label{Kxy-lem}
  For every $a\in S^d_{1,1}(\Rn\times\Rn)$ the kernel 
$K(x,y)$ is $C^\infty$ for $x\ne y$. 
\end{lem}
\begin{proof}
For $N$ so large that $d+|\beta|+|\alpha|-2N<-n$,
\begin{equation}
 |z|^{2N}D^\beta_xD^\alpha_z\cal F_{\eta\to z}(a(x,\eta))=
\cal F_{\eta\to z}(\lap^N_\eta(\eta^\alpha D^\beta_x a(x,\eta))) 
  \label{zK-eq}
\end{equation}
is a continuous function, so 
any derivative of $K$ is so for $x\ne y$.  
\end{proof}

Instead the middle of \eqref{aK-eq} gives a convenient way to
prove that every type $1,1$-operator extends to $\cal F^{-1}\cal E'(\Rn)$,
ie to the space of tempered distributions with compact spectrum. This result
was first observed in \cite{JJ05DTL}, but the following argument
should be interesting for its simplicity.
When $v\in C^\infty_0(\Rn)$
and $u\in \cal F^{-1}C^\infty_0(\Rn)$ then \eqref{aK-eq} gives
\begin{equation}
  \dual{a(x,D)u}{v}= \dual{v(x)\otimes\hat u(\eta)}
   {\tfrac{e^{\im x\cdot\eta}}{(2\pi)^n}a(x,\eta)}_{\cal E'\times C^\infty}.
  \label{axD'-eq}
\end{equation}
This suggests to introduce $A\colon \cal F^{-1}\cal E'(\Rn)\to
C^\infty(\Rn)$ given by
\begin{equation}
  Au(x)= \dual{\hat u}{\tfrac{e^{\im \dual{x}{\cdot}}}{(2\pi)^n}a(x,\cdot)}.
  \label{Adual-eq}
\end{equation}
That $Au$ is in $C^\infty$ is a standard fact used eg in the construction of
tensor products on $\cal E'(\R^{n'})\times\cal E'(\R^{n''})$; cf 
\cite[Th.~5.1.1]{H}. By definition of the tensor product
of arbitrary $v$, $\hat u\in \cal E'(\Rn)$, they should act successively
on the $C^\infty$-function, so for $v$ in $C^\infty_0(\Rn)$,
\begin{equation}
  \dual{v\otimes\hat u}{e^{\im\dual{\cdot}{\cdot}}a\cdot(2\pi)^{-n}}
 =
  \dual{v}{Au} =\int_{\Rn} v(x)Au(x)\,dx.
\end{equation}
This and \eqref{axD'-eq} 
gives $Au=a(x,D)u$ for every 
$u\in \cal F^{-1}C^\infty_0=\cal S\cap\cal F^{-1}\cal E'$;
hence $a(x,D)$ and $A$ are compatible.
Therefore $a(x,D)$ extends to a map
\begin{equation}
  a(x,D)\colon \cal S(\Rn)+\cal F^{-1}\cal E'(\Rn)\to C^\infty(\Rn)
\end{equation}
by setting $a(x,D)u=a(x,D)v+Av'$ when $u=v+v'$ for $v\in \cal S$ and
$\cal Fv'\in \cal E'$ (if $0=v+v'$, clearly $v=-v'$ is in 
$\cal F^{-1}C^\infty_0$, hence gives identical images, ie $a(x,D)v+Av'=0$).

\bigskip

By the duality of $\cal E'$ and $C^\infty$, the right-hand side of
\eqref{Adual-eq} should be 
calculated by multiplying $a(x,\eta)$ by a cut-off function $\chi(\eta)$
equalling $1$ on a neighbourhood of $\supp\cal Fu$. 
The resulting symbol $\chi(\eta)a(x,\eta)$ is clearly in
\begin{equation}
 S^{-\infty}(\Rn\times\Rn)
:=\bigcap_{d,\rho,\delta}S^{d}_{\rho,\delta}(\Rn\times\Rn).
  \label{S-8-eq}
\end{equation} 
A systematic exploitation of localisations $\chi(\eta)a(x,\eta)$ 
is found in the next section.

\subsection{Extension by spectral localisation}  \label{loc-ssect}
For type $1,1$-operators, this section gives a first extension, 
based on cut-off techniques and arguments from algebra. 
The latter are trivial, but important
for several compatibility questions that are treated here. 

\bigskip

Let $\cal S'_{\Sigma}(\Rn)$ denote the closed 
subspace of distributions with
spectrum in a given open set $\Sigma\subset\Rn$, ie
\begin{equation}
  \cal S'_{\Sigma}(\Rn)=\bigl\{\,u\in \cal S'(\Rn) \bigm| \supp\hat
u\subset\Sigma\,\bigr\}.
  \label{S'O-eq}
\end{equation}
Clearly the intersection $\cal S_{\Sigma}(\Rn):=
\cal S(\Rn)\cap\cal S'_{\Sigma}(\Rn)$ is dense in $\cal
S'_{\Sigma}(\Rn)$. 

As a basic assumption in this section, $a(x,\eta)\in S^\infty_{1,1}$ 
should have the properties of a more `well-behaved' symbol class 
$S$ as $\eta$ runs through a given open set $\Sigma\subset\Rn$.
It would then be natural, and necessary, to extend $a(x,D)$
to every $u\in \cal S'_{\Sigma}(\Rn)$ by letting it act 
as an operator with symbol in the class $S$.

To turn this idea into a definition, an arbitrary linear subspace 
$S \subset\cal S'(\R^{2n})\cap C^\infty(\R^{2n})$ 
will in the following be called a
\emph{standard} symbol space if, for every $b\in S$,
the integral in \eqref{axDu-eq} gives an operator 
$\OP(b)\colon \cal S\to\cal S$
which  extends to a continuous linear map
\begin{equation}
  \OP(b)\colon \cal S'(\Rn)\to\cal S'(\Rn).
  \label{OPb-eq}
\end{equation}
(Such an extension is unique, so the notation need not relate $\OP(b)$ to
the choice of $S$. To avoid confusion, the type $1,1$ operator under
extension is usually denoted $a(x,D)$.)
An example could be $S=S^d_{\rho,\delta}$ with
$(\rho,\delta)\ne(1,1)$; whilst
$\OP(b)$ could be the extension to $\cal S'$ of $b(x,D)$ 
given by the adjoint of 
$b^*(x,D)\colon \cal S(\Rn)\to\cal S(\Rn)$, that in its turn is defined from
the adjoint symbol $b^*(x,\xi)=e^{\im D_x\cdot D_\xi}\bar b(x,\xi)$.

Using this, $a(x,D)$ can be
extended if the symbol $a\in S^\infty_{1,1}$ is \emph{locally} in a standard
symbol space $S$ in an open set $\Sigma\subset\Rn$.
Specifically this means that for every
closed set $F\subset\Sigma$ there exists 
a cut-off function $\chi\in  C^\infty_{\b}(\Rn)$, not necessarily supported
by $\Sigma$, such that 
\begin{equation}
  \chi\equiv1 \text{ on a neighbourhood of } F,\qquad
  \chi(\eta)a(x,\eta)\in  S.
  \label{chi-cnd}
\end{equation}
Instead of $a(x,\eta)\chi(\eta)$, the slightly more correct
$a(1\otimes\chi)$ is often preferred in the sequel. 

\begin{prop}   \label{aloc-prop}
For each symbol $a\in S^\infty_{1,1}(\Rn\times\Rn)$ that is locally in $S$ 
in an open set $\Sigma\subset\Rn$, there is defined a map 
$\cal S'_{\Sigma}(\Rn)\to \cal S'(\Rn)$ by
\begin{equation}
  a(x,D)u=\OP(a(1\otimes\chi))(u),
  \label{aloc-eq}
\end{equation}
which has the same value for all $\chi\in C^\infty_{\b}(\Rn)$ 
satisfying \eqref{chi-cnd} for $F=\supp\cal Fu$.
The map is compatible with $a(x,D)\colon \cal S(\Rn)\to\cal S(\Rn)$.
\end{prop}
\begin{proof}
Let $u\in \cal S'_{\Sigma}(\Rn)$.
By \eqref{chi-cnd} and \eqref{OPb-eq}, $\OP(a(1\otimes\chi))$ is 
defined on $\cal S'\ni u$. If $\chi_1$ is another such function, 
$a(x,\eta)(\chi(\eta)-\chi_1(\eta))$ is in the vector space $S$ and  
equals $0$ for $\eta$ in some open
set $\Sigma_1\supset \supp u$, so that by density of $\cal S_{\Sigma_1}$ in
$\cal S'_{\Sigma_1}$,
\begin{equation}
  0= \OP(a(1\otimes(\chi-\chi_1)))u.
  \label{acc1-eq}
\end{equation}
Therefore \eqref{aloc-eq} is independent of the choice of $\chi$, so
the map $\OP(a(1\otimes\chi))u$ is defined; it equals
$a(x,D)u$ for every 
$u\in \cal S\cap\cal S'_{\Sigma}$ by \eqref{axDu-eq}.
\end{proof}

The compatibility in Proposition~\ref{aloc-prop} 
gives of course a map on the algebraic subspace 
$\cal S(\Rn)+\cal S'_{\Sigma}(\Rn)\subset\cal S'(\Rn)$;
cf \eqref{aSSO-eq}. But more holds:

\begin{thm}   \label{aSSO-thm}
For every $a\in  S^\infty_{1,1}(\Rn\times\Rn)$ that 
in an open set $\Sigma\subset\Rn$ is locally in a standard symbol space
$S$ (cf \eqref{OPb-eq}),
the operator $a(x,D)$ extends to a linear map
\begin{align}
\cal S(\Rn)+\cal S'_{\Sigma}(\Rn) 
  &\xrightarrow[]{a(x,D)} \cal S'(\Rn) \quad\text{given by}
\\
  a(x,D)u&= a(x,D)v+\OP(a(1\otimes \chi))v'
  \label{aSSO-eq}  
\end{align}
whenever $u=v+v'$ for $v\in \cal S(\Rn)$, $v'\in \cal
S'_{\Sigma}(\Rn)$; hereby
$\chi\in C^\infty_{\b}(\Rn)$ can be any function 
fulfilling \eqref{chi-cnd} for $F=\supp\cal Fv'$.
The extension is uniquely determined by coinciding with \eqref{OPb-eq} on
$\cal S'_\Sigma(\Rn)$.
\end{thm}
\begin{proof}
For uniqueness, let
$\widetilde{\OP}(a)$ be any extension agreeing with
\eqref{OPb-eq} on $\cal S'_{\Sigma}(\Rn)$. Then linearity 
gives, for any splitting $u=v+v'$ and $\chi$ as in the theorem, that
\begin{equation}
  \widetilde{\OP}(a)u=\widetilde{\OP}(a)v+\widetilde{\OP}(a)v'
  =a(x,D)v+\OP(a)v'
  =a(x,D)v+\OP(a(1\otimes\chi))v'.
\end{equation}

To show that \eqref{aSSO-eq} actually defines the desired map,
suppose $u=v+v'=w+w'$ for some $v,w\in \cal S$ and
$v',w'\in \cal S'_{\Sigma}$. Applying $a(x,D)$ to $v-w$ and
$\OP(a(1\otimes\chi))$ to $w'-v'$,
with $\chi$ taken so that 
$\chi\equiv1$ on a neighbourhood of $F=\supp \cal Fv'\cup
\supp\cal F w'\supset\supp\cal F(w'-v')$, 
it follows from the compatibility in
Proposition~\ref{aloc-prop} and linearity that, for $\chi=\chi_1=\chi_2$,
\begin{equation}
 a(x,D)v+\OP(a(1\otimes\chi_1))v'=
 a(x,D)w+\OP(a(1\otimes\chi_2))w'.
\end{equation}
By Proposition~\ref{aloc-prop} one can then pass to arbitrary
$\chi_1$, $\chi_2$ equalling $1$ around $\supp\cal Fv'$, respectively
$\supp\cal Fw'$, without changing the left and right-hand sides.
This means that \eqref{aSSO-eq} gives a map, for
$a(x,D)v+\OP(a(1\otimes \chi))v'$ 
is independent of the splitting $u=v+v'$ and of the corresponding  
choice of $\chi$;
thence linear dependence on $u$ follows too. 
\end{proof}

Theorem~\ref{aSSO-thm} gives a basic
extension of type $1,1$-operators, that
could have been a definition (justified by the given arguments).
When $a\in  S^\infty_{1,1}$ happens to be in $S$ too, 
then $\chi\equiv1_{\Rn}$ and $v=0$ yields 
$a(x,D)u=\OP(a)u$, so the definition \eqref{aSSO-eq} 
gives back the $\cal S'$-continuous operators with symbols in $S$.

Before these questions are pursued, the construction's 
dependence on $S$ and $\Sigma$ is
investigated. 

\begin{prop}  \label{SSt-prop}
Let $S$ and $\tilde S$ be standard symbol spaces, and
let $a\in S^\infty_{1,1}$ be locally in $S$ in some open set
$\Sigma\subset\Rn$ and also locally in
$\tilde S$ in an open set $\tilde \Sigma$. Then the induced maps
\begin{equation}
 a(x,D)\colon \cal S+\cal S'_{\Sigma}\to\cal S', 
 \qquad
 a(x,D)\colon \cal S+\cal S'_{\widetilde  \Sigma}\to\cal S'
\end{equation}
are compatible when either $\Sigma$ has the property
that $\chi$ in \eqref{chi-cnd} for every $F$ can be taken with
$\supp\chi\subset\Sigma$, or $\tilde \Sigma$ has the analogous property. 
\end{prop}
\begin{proof}
One can reduce to the case $S=\tilde S$ by introducing the subspace
$S+\tilde S\subset\cal S'(\R^{2n})$: 
for every $b\in S$, $\tilde b\in \tilde S$ the definition by the usual
integral shows that
\begin{equation}
  \OP( b+\tilde b)=
  \OP(b)+\OP(\tilde b) \quad\text{on}\quad \cal S(\Rn).
\end{equation}
Here $\OP(b+\tilde b)$ extends to a continuous, linear map 
$\cal S'(\Rn)\to \cal S'(\Rn)$ since the right-hand side does so.
If $b+\tilde b=b_1+\tilde b_1$ for $b_1\in S$, $\tilde b_1\in \tilde S$,
both $\OP(b+\tilde b)$, $\OP(b_1+\tilde b_1)$ extend to $\cal
S'$, where they coincide as they do so on $\cal S$. 
Hence
every $b+\tilde b$ in $S+\tilde S$ gives an unambiguously defined
operator on $\cal S'(\Rn)$, as required in \eqref{OPb-eq}; ie $S+\tilde
S$ is a standard space. 

Let $u=v+w=\tilde v+\tilde w$ for some $v,\tilde v\in \cal S$, $w\in
\cal S'_{\Sigma}$ and $\tilde w\in \cal S'_{\widetilde{\Sigma}}$.
By the last assumption there exists eg $\varphi\in C^\infty_{\b}(\Rn)$
such that $\supp\varphi\subset\tilde \Sigma$, $\varphi\equiv1$ on a
neighbourhood of $\tilde F$ and $a(1\otimes\varphi)\in \tilde S$.
In particular $1-\varphi=0$ around $\tilde F$ so
\begin{equation}
  u=\varphi(D)(v+w)+(1-\varphi(D))(\tilde v+\tilde w)
   =\tilde v+\varphi(D)(v-\tilde v)+\varphi(D)w.
\end{equation}
While $v'=\tilde v+\varphi(D)(v-\tilde v)$ is in $\cal S(\Rn)$,
the term  $\varphi(D)w$ is in $\cal S'_{\Sigma\cap\tilde
\Sigma}(\Rn)$. 
By taking $\psi=1$ in a neighbourhood of $\supp\cal
F\varphi(D)w\subset\Sigma\cap\tilde \Sigma$, 
it is clear that one gets
\begin{equation}
  a(x,D)u=a(x,D)v'+\OP(a(1\otimes\psi))\varphi(D)w
\end{equation}
by application of Theorem~\ref{aSSO-thm} both for $\cal
S'_{\Sigma}$ and $\cal S'_{\tilde \Sigma}$.
\end{proof}

As a simple application for $\Sigma=\Rn$,  every
$u\in \cal F^{-1}\cal E'(\Rn)$ is in
$\cal S'_{\Sigma}=\cal S'(\Rn)$; and $a$ is locally in
$S^{-\infty}$ since
$b(x,\eta)=a(x,\eta)\chi(\eta)$ is in $S^{-\infty}$
for every $\chi\in C^\infty_0(\Rn)$, in particular when $\chi=1$ around
$\supp \cal F u$.
Therefore Theorem~\ref{aSSO-thm} yields a unique extension of $a(x,D)$
to a linear map $\cal S(\Rn)+\cal F^{-1}\cal E'(\Rn)$.
(Proposition~\ref{SSt-prop} shows that one can replace the reference to 
$S^{-\infty }$ by eg
$S^\infty _{1,0}$ or let $\Sigma$ depend on $u$ without changing the image
$a(x,D)u$.) 

This approach is more elementary than \eqref{axD'-eq} ff. 
In addition it gives that $a(x,D)$ maps $\cal F^{-1}\cal E'$ 
into $\cal O_M(\Rn)$.
Recall that every $a(x,D)$ in $\OP(S^{-\infty}(\Rn\times\Rn))$ 
is a map $\cal S'\to \cal O_M$  (cf \cite[Cor.~3.8]{SRay91}), 
since if $u\in \cal S'$, then
$(1+|x|^2)^{-N}u\in H^{-N}$ for some $N>0$ and  every commutator
$[D^\alpha a(x,D),(1+|x|^2)^N]$ is by inspection in $\OP(S^{-\infty})$.
These facts imply the next result.

\begin{cor}   \label{aFE-cor}
Every operator $a(x,D)$ with symbol in $S^\infty_{1,1}$ 
extends uniquely to a map
\begin{equation}
  a(x,D)\colon \cal S(\Rn)+\cal F^{-1}(\cal E'(\Rn))\to \cal O_M(\Rn),
  \label{aFE-eq}
\end{equation}
which is given by Theorem~\ref{aSSO-thm} with $S=S^{-\infty}$.
\end{cor}

Notice that the corollary's statement is purely algebraic, since
continuity properties are not involved in \eqref{aFE-eq}.
Similarly one has another extension result.

\begin{prop}
  \label{aSV-prop}
If $a\in S^\infty_{1,1}$
is locally in the symbol class $S^d_{1,0}(\Rn\times\Rn)$ in an open cone
$V\subset\Rn$ (ie $t\eta\in V$ for all
$t>0$ and $\eta\in V$), then
\eqref{aSSO-eq} yields a unique extension
\begin{equation}
  a(x,D)\colon \cal S+\cal S'_{V}\to \cal S'.
  \label{aSV-eq}
\end{equation}
\end{prop}

If some $a\in S^\infty_{1,1}$ satisfies the hypotheses of 
Proposition~\ref{aSV-prop}, it follows from Proposition~\ref{SSt-prop}
that the two extensions in \eqref{aFE-eq}--\eqref{aSV-eq} 
are compatible with one another. 

\begin{exmp} By Corollary~\ref{aFE-cor}, the domain of every $a(x,D)$
in $\OP (S^{\infty}_{1,1})$ contains polynomials
$\sum_{|\alpha|\le k}c_\alpha x^\alpha$, 
as their spectra equal $\{0\}$, and eg also the $C^\infty$-functions 
\begin{equation}  
  x^\alpha e^{\im x\cdot z}
  =(2\pi)^n\cal F^{-1}(\overline{D}_\eta^\alpha\delta_z(\eta));
\qquad
  \tfrac{\sin x_1}{x_1}\dots\tfrac{\sin x_n}{x_n}
  = \pi^n \cal F^{-1}1_{[-1;1]^n}.
\end{equation}
\end{exmp}

\section{Definition by Vanishing Frequency Modulation}   \label{gdef-sect}
The full extension of type $1,1$-operators is given here by
means of a limiting procedure.

\bigskip

To define $a(x,D)u$ in general for $a\in S^d_{1,1}(\Rn\times\Rn)$,
$d\in \R$, and
suitable $u\in \cal S'(\Rn)$, it is convenient for an arbitrary $\psi\in
C^\infty_0(\Rn)$ with $\psi\equiv1$ in a neighbourhood of the origin to
introduce the following notation, 
with $\psi_m(\xi):=\psi(2^{-m}\xi)$, 
\begin{align}
  u^m&=\cal F^{-1}(\psi_m\hat u)=\psi_m(D)u,
  \\
  a^m(x,\eta)&=\cal F^{-1}_{\xi\to x}(\psi_m(\xi)\cal F_{x\to\xi}a(\xi,\eta))
 =\psi_m(D_x)a(x,\eta).  
\end{align}
This is referred to as a frequency modulation of $u$ and of $a(x,\eta)$ with
respect to $x$; the full frequency modulation of $a$ will be
$a^m(x,\eta)\psi_m(\eta)$, ie $a^m(1\otimes \psi_m)$. 
Since $a^m$ is in $S^\infty_{1,1}$ by Lemma~\ref{am-lem}, 
the compact support of $\psi_m$ gives that 
\begin{equation}
  a^m(1\otimes\psi_m)\in S^{-\infty}.  
  \label{ammu'-eq}
\end{equation}
Hence $\OP(a^m(1\otimes\psi_m))$ is defined on $\cal
S'(\Rn)$, and since
$\lim_{m\to\infty}a^m(1\otimes\psi_m)=a$ holds in $S^{d+1}_{1,1}$,
it should be natural to make a tentative definition of $a(x,D)$ as
\begin{equation}
  a(x,D)u=\lim_{m\to\infty} \OP(a^m(1\otimes\psi_m))u.
  \label{ammu-eq}
\end{equation}
Rougly speaking, this means approximation of the distribution $u$ by
elements of $\cal S(\Rn)$ is
replaced by regularisation of the symbol $a$.
Some difficulties that might appear in this connection are dealt with
in the formal

\begin{defn}
  \label{gdef-defn}
The pseudo-differential operator $a(x,D)u$ is defined as the limit
in \eqref{ammu-eq} for those $a\in
S^d_{1,1}(\Rn\times\Rn)$ and $u\in \cal S'(\Rn)$ for which the limit 
\begin{itemize}
  \item exists in the topology of $\cal D'(\Rn)$
     for every $\psi\in C^\infty_0(\Rn)$ 
     equalling $1$ in a neighbourhood of the origin, and
  \item is independent of such $\psi$.
\end{itemize}
\end{defn}
To show that $a(x,D)$ extends the operator 
defined on $\cal S(\Rn)$ by \eqref{axDu-eq}, it suffices to combine
Lemma~\ref{am-lem} with \eqref{bilSS-eq}.
As shown below, Definition~\ref{gdef-defn} also 
gives back both the usual operator $\OP(a)$ if $a$ is eg of type $1,0$
and the extensions in Section~\ref{basic-sect}.

\bigskip

As an elementary observation, by using the definition for a fixed $a\in
S^\infty_{1,1}$ and by the calculus of limits,  
the operator is defined for $u$ in a subspace
of $\cal S'(\Rn)$. This will be denoted by $D(a(x,D))$, or $D(A)$ if
$A:=a(x,D)$, in the following. 

Clearly $D(A)\supset\cal S(\Rn)$, so  
$A$ is a densely defined and linear operator from $\cal
S'(\Rn)$ to $\cal D'(\Rn)$ (borrowing terminology from unbounded operators
in Hilbert spaces). 
This description cannot be improved much in general,
for by Lemma~\ref{cex-lem},
$a(x,\eta)$ can be chosen so that $A$ with $D(A)=\cal S(\Rn)$ 
is unclosable.
But one has

\begin{prop}
  \label{ab-prop}
For $a$, $b$ in $S^\infty_{1,1}(\Rn\times\Rn)$ the following
properties are equivalent:
\begin{rmlist}
    \item $a(x,\eta)=b(x,\eta)$ for all $(x,\eta)\in \R^{2n} $;
  \item $a(x,D)=b(x,D)$ as operators from $\cal S'(\Rn)$ to $\cal D'(\Rn)$;
  \item $a(x,D)u=b(x,D)u$ for every $u\in \cal S(\Rn)$;
  \item the distribution kernels fulfil $K_a=K_b$.
\end{rmlist}
In particular the map $a\mapsto a(x,D)$ is a bijection
$S^d_{1,1}\leftrightarrow \OP(S^d_{1,1})$; and the operator $a(x,D)$ is
completely determined by its action on the Schwartz space.
\end{prop}

The last property is perhaps not obvious from the outset,
because, in general, there is neither density of $\cal S\subset D(a(x,D))$
nor continuity of $a(x,D)$ to appeal to. However it follows at once,
as it is straightforward to see that
\rom{(i)}$\implies$\rom{(ii)}$\implies$\rom{(iii)}$\implies$\rom{(iv)}$\implies$\rom{(i)}.

The following notion is very convenient for the analysis of $a(x,D)$:

\begin{defn}
  \label{stable-defn}
A standard symbol space $S$ on $\Rn\times\Rn$ is said to be 
\emph{stable under vanishing frequency modulation} 
if in addition to \eqref{OPb-eq},
\begin{rmlist}
  \item   \label{stabl1-cnd}
$b^m(1\otimes\psi_m)(x,\eta)=\psi(2^{-m}D_x)b(x,\eta)\psi_m(\eta)$, 
is in $S$ for every $b\in S$, $m\in\N$,
and every $\psi\in C^\infty_0(\Rn)$ equalling $1$ near the origin,
  \item    \label{stabl2-cnd}
for every $u\in \cal S'(\Rn)$, and $\psi$ as above, 
\begin{equation}
  \OP(b^m(1\otimes\psi_m))u\to \OP(b)u 
  \quad\text{in $\cal D'(\Rn)$ for}\quad m\to\infty.
  \label{bilim-eq}
\end{equation}
\end{rmlist}
For short $S$ and the operator class $\OP(S)$ are 
then said to be \emph{stable}.
\end{defn}
Note that \eqref{stabl1-cnd} 
requires 
the operator class $\OP(S)$ to be invariant under full frequency modulation;
whereas \eqref{stabl2-cnd} requires $\OP(S)$ to be invariant under
\emph{vanishing} frequency modulation in the sense that the limit gives back
the original operator $\OP(b)$.

Although $S^d_{1,1}$ is not a standard space, 
$\OP(S^\infty_{1,1})$ is also said to be stable, 
as \eqref{bilim-eq} holds by definition
for every $u$ in $D(b(x,D))$, $b\in S^\infty _{1,1}$. 
Other stable spaces exist as well:

\begin{prop}  \label{Stan-prop}
Every  $S^d_{\rho,\delta}(\Rn\times\Rn)$ 
with $\rho>\delta$ for $\rho,\delta\in [0,1]$ is a stable symbol space.
Moreover,  \eqref{bilim-eq} holds in the $\cal S'$-topology. 
\end{prop}
\begin{proof}
By Lemma~\ref{am-lem} condition \eqref{stabl1-cnd} is satisfied, and
$\lim b^m(1\otimes\psi_m)= b$ in
$S^{d'}_{\rho,\delta}$, $d'>d+\delta$.
As $\OP(b)$ is the adjoint of $b^*(x,D)=\OP(\exp(\im
D_{x}\cdot D_{\eta})\overline{b})$, 
each $\varphi\in \cal S(\Rn)$ gives
\begin{equation}
  \scal{\OP(b^m(1\otimes\psi_m))u-\OP(b)u}{\varphi}
  = \scal{u}{\OP(e^{\im D_x\cdot D_\eta}
     (\overline{b^m}(1\otimes\overline{\psi}_m)-\overline{b}))\varphi}
  \xrightarrow[m\to\infty]{~} 0,
\end{equation}
since 
passage to adjoint symbols $b\mapsto b^*$ is continuous $S^d_{\rho,\delta}\to
S^d_{\rho,\delta}$ for $\rho>\delta$.
\end{proof}

Proposition~\ref{Stan-prop} makes 
the definition of $a(x,D)$ by vanishing
frequency modulation look natural.
To analyse the consistency questions in general, 
it is recalled that $\OP(a)$ is defined on
$\cal S(\Rn)$ by the integral \eqref{axDu-eq} if $a$ is in a 
standard space $S$ or in $S=S^\infty_{1,1}$. And for a standard space $S$,
$\OP(a)$ extends uniquely to a continuous linear map 
on $\cal S'(\Rn)$.

Let now $a\mapsto \OPT(a)$ be an arbitrary assignment 
such that $\OPT(a)$, for each $a\in  S^\infty_{1,1}$, is a
linear operator from $\cal S'(\Rn)$ to $\cal D'(\Rn)$. 
Then the maps $\OP$ and $\OPT$ are
compatible on a standard symbol space $S$ if
$D(\OPT(a))=\cal S'(\Rn)$ for every $a\in S\cap S^\infty_{1,1}$ and
\begin{equation}
  \OPT(a)u=\OP(a)u\quad\text{for all}\quad u\in \cal S'(\Rn).
  \label{OPT-eq}
\end{equation}
Moreover, $\OPT$ and $\OP$ are called \emph{strongly} compatible on $S$ if, 
whenever $a$ is in $S^\infty_{1,1}$ and belongs to $S$ locally in some open set
$\Sigma\subset\Rn$, it will hold that
$\cal S'_\Sigma(\Rn)\subset D(\OPT(a))$ and
\begin{equation}
  \OPT(a)u=\OP(a(1\otimes \chi))u\quad\text{for all}\quad 
  u\in \cal S'_\Sigma(\Rn).
  \label{OPT'-eq}
\end{equation}
Hereby $\chi\in C^\infty_{\b}(\Rn)$ should 
fulfil \eqref{chi-cnd} for  
$F=\supp \hat u$ and $a(1\otimes\chi)\in S$.
(The right-hand side of \eqref{OPT'-eq} 
makes sense because of $\chi$, but it does not
depend on $\chi$ since $S$ is standard.) 
Taking $\Sigma=\Rn$ and
$\chi\equiv1$, strong compatibility clearly implies compatibility.

As an example Corollary~\ref{aFE-cor} shows that, if the preliminary extension 
of Section~\ref{loc-ssect} is written $\OPT$, then $\OPT(a)$ is strongly
compatible with $\OP$ on $S^{-\infty}$. More generally
Theorem~\ref{aSSO-thm} gives strong compatiblity of $\OPT$ with $\OP$ on
every standard symbol class $S$.

The following theorem shows that $a(x,D)$ given by
Definition~\ref{gdef-defn} contains every extension provided by
Theorem~\ref{aSSO-thm} when $S$ is stable.

\begin{thm}
  \label{consistency-thm}
Let $a\in S^\infty_{1,1}$ and $\Sigma\subset\Rn$ be an open set such that
$a$ locally in $\Sigma$ belongs to a stable symbol class $S$ (such as
$S^d_{\rho,\delta}$ for $\rho>\delta$). 
Then every $u\in \cal S(\Rn)+\cal S'_{\Sigma}(\Rn)$
belongs to the domain $D(a(x,D))$ given by Definition~\ref{gdef-defn}.
Moreover,
\begin{equation}
  a(x,D)u=\OP(a)v+\OP(a(1\otimes\chi))v'
  \label{uvv'-eq}
\end{equation}
whenever $u$ is split as $u=v+v'$ for $v\in \cal S(\Rn)$, $v'\in \cal
S'_{\Sigma}(\Rn)$, and $\chi\in C^\infty_{\b}(\Rn)$ fulfils
\eqref{chi-cnd} for $F=\supp\cal F v'$.  
\end{thm}

\begin{proof} 
Let $u\in \cal S'_{\Sigma}$.
Since $a$ is locally in $S$ in $\Sigma$ one can take $\chi$ as in the theorem,
so that $a(1\otimes\chi)\in S$.
Using that $S$ in particular is a standard space, approximation of 
$\hat u$ from $C^\infty _0$ gives 
$\OP(a^m(x,\eta)(1-\chi(\eta))\psi_m(\eta))u=0$.
Now \eqref{bilim-eq}
applies, since $S$ is stable; and multiplication by
$\chi(\eta)$ and $\psi_m(D_x)$ commute in $\cal S'(\Rn\times\Rn)$, so
\begin{equation}
  \OP(a(1\otimes\chi))u
  =\lim_{m}\OP(\psi_m(\eta)\chi(\eta)\psi_m(D_x)a(x,\eta))u
  =\lim_{m}\OP(a^m(1\otimes\psi_m))u=a(x,D)u.
  \label{astable-eq}
\end{equation}
This shows that $\cal S'_{\Sigma}(\Rn)\subset D(a(x,D))$. 
And for $u\in \cal S(\Rn)$ it is seen already from
\eqref{bilSS-eq} that $a^m(x,D)u^m\to a(x,D)u$ in $\cal S(\Rn)$ for
$m\to\infty$. 
 
Since $a(x,D)$ is linear by Definition~\ref{gdef-defn}, 
it follows that every $u$ in $\cal S+\cal
S'_{\Sigma}$ belongs to $D(a(x,D))$ and that \eqref{uvv'-eq} holds.
In particular the last statement that \eqref{uvv'-eq} is independent of $v$,
$v'$ and $\chi$ is implied by this (and by Theorem~\ref{aSSO-thm}).
\end{proof}

\begin{rem}   \label{consistency-rem}
It is noteworthy that Theorem~\ref{consistency-thm} resolves
a dilemma resulting from application of $a(x,D)\in \OP(S^\infty_{1,1})$ 
to  $u\in \cal F^{-1}(\cal E'(\Rn))$: 
then $a(x,D)u$ can be calculated by using 
both Corollary~\ref{aFE-cor} and Definition~\ref{gdef-defn}. But by taking
$S=S^{-\infty}$, 
Theorem~\ref{consistency-thm} entails that the two methods 
give the same result.
\end{rem}

It follows from Theorem~\ref{consistency-thm} that the assumptions on
$\Sigma$ and $\tilde\Sigma$ are unnecessary in Proposition~\ref{SSt-prop}
in case $S$ is stable (this emphasises the advantage of using 
vanishing frequency modulation). 
As a reformulation of Theorem~\ref{consistency-thm} one has

\begin{cor}
  \label{stcomp-cor}
The operator $a(x,D)$ given for $a\in S^{\infty}_{1,1}(\Rn\times\Rn)$ by
Definition~\ref{gdef-defn} is strongly compatible with $\OP$ on every
stable symbol space $S$. In particular $a(x,D)u=\OP(a)u$ holds for every
$u\in \cal S'(\Rn)$ when $a\in S^{\infty}_{\rho,\delta}(\Rn\times\Rn)$ for
some $\delta<\rho$.
\end{cor}

As a special case $a(x,D)$ gives back $\OP(a)$ on
$S^{-\infty}$. This may also be shown by verifying
\eqref{OPT-eq} directly, but one can only simplify \eqref{astable-eq}
slightly by taking $\chi\equiv1$ on $\Rn$.
The various consistency results obtained in this section can be summed
up thus: 
\begin{cor}
  \label{consistency-cor}
Let $a(x,D)$ be given by Definition~\ref{gdef-defn} for 
$a\in  S^\infty_{1,1}$.
Then $a(x,D)u$ equals the integral in \eqref{axDu-eq} for $u\in
\cal S(\Rn)$ or the extension in Corollary~\ref{aFE-cor} for
every $u\in \cal F^{-1}\cal E'(\Rn)$;
and it coincides with the extension of $\OP(a)$ to $\cal S'(\Rn)$ 
if $a$ is in 
$S^d_{\rho,\delta}(\Rn\times\Rn)$ for some $\rho>\delta$.
\end{cor}

To characterise the operators provided by Definition~\ref{gdef-defn}, it is
convenient to ignore that the compatibility of $a(x,D)$ is strong (cf
Corollary~\ref{stcomp-cor}). Indeed, the map $a\mapsto a(x,D)$ 
is simply the \emph{largest} 
compatible extension stable under vanishing frequency modulation:

\begin{thm}   \label{uniq-thm}
The operator $a(x,D)$ given by Definition~\ref{gdef-defn} is one among the
operator  assignments $a\mapsto\OPT(a)$, $a\in S^\infty_{1,1}(\Rn\times\Rn)$ 
with the properties that
\begin{rmlist}
  \item   \label{scmtb-cnd}
$\OPT(\cdot)$ is compatible with $\OP$ on $S^{-\infty}$ (cf \eqref{OPT-eq});
  \item  \label{stabl-cnd} 
each operator $\widetilde{\OP}(b)$ 
is stable under vanishing frequency modulation, ie
$\OPT(b)u=\lim_{m\to\infty}\OPT(b^m(1\otimes \psi_m))u$ 
for every $u\in D(\widetilde{\OP}(b))$ and
$b\in S^\infty_{1,1}$. 
\end{rmlist}
And moreover, whenever $\OPT$ is such a map, 
then $\OPT(a)\subset a(x,D)$ 
for every $a\in S^\infty_{1,1}$.
\end{thm}

Note that \eqref{stabl-cnd} makes sense because
$\OPT(b^m(1\otimes \psi_m))$ in view of 
\eqref{scmtb-cnd} is defined on all of $\cal S'$.

\begin{proof} 
Let $a\mapsto \widetilde{\OP}(a)$ be any map fulfilling \eqref{scmtb-cnd}
and \eqref{stabl-cnd}; such maps exist since $a\mapsto a(x,D)$ was seen
above to have these properties. If $u\in D(\OPT(a))$ it follows from 
\eqref{scmtb-cnd} that
\begin{equation}
  \OP(a^m(1\otimes\psi_m))u=\OPT(a^m(1\otimes\psi_m))u.
\end{equation}
Here the right-hand side converges to
$\OPT(a)u$ by \eqref{stabl-cnd};
since $\psi$ is arbitrary this means
$\OPT(a)u=a(x,D)u$. Hence $\OPT(a)\subset a(x,D)$.
\end{proof}

\bigskip

This section is concluded with a few remarks on the practical aspects of
Definition~\ref{gdef-defn}. 
From the integral in \eqref{axDu-eq}, one would at once infer
the following \emph{alter egos} for the full frequency modulation of $a(x,D)$: if $\chi\in C^\infty_0(\Rn)$ fulfils
$\chi=1$ around $\supp\psi_m$, then
\begin{equation}
  a^m(x,D)u^m=\OP(a^m(x,\eta)\chi(\eta))\cal F^{-1}(\psi_m\hat u)
 =\OP(a^m(x,\eta)\psi_m(\eta))u.
  \label{alter-eq}
\end{equation}
However, these identities hold also for more general cut-off functions $\chi$.
\begin{lem}
  \label{alter-lem}
For every $a\in S^\infty_{1,1}$, $u\in \cal S'$ 
and every $\psi\in C^\infty_0$ with $\psi=1$ near the origin, the formula
\eqref{alter-eq} holds for all $m$ and all $\chi\in C^\infty_{\b}$ for
which $\chi=1$ in a neighbourhood of $F=\supp(\psi_m\hat u)$ and 
$a(1\otimes\chi)\in S^{-\infty}$. 
\end{lem}
\begin{proof}
The last part of \eqref{alter-eq} follows from
\eqref{axDu-eq} if $u$ is a Schwartz 
function, hence for all $u\in \cal S'(\Rn)$ since both $a^m(1\otimes\psi_m)$
and $a^m(1\otimes\chi)$ belong to $S^{-\infty}$. 
Since $u^m\in \cal F^{-1}\cal E'$ and $a^m(x,\eta)\in S^\infty _{1,1}$,
Corollary~\ref{consistency-cor} shows that $a^m(x,D)(u^m)$ can be calculated
by the extension in Section~\ref{loc-ssect}; then Theorem~\ref{aSSO-thm}
gives the left-hand side of \eqref{alter-eq}.
\end{proof}

In view of \eqref{alter-eq},
one could alternatively have defined $a(x,D)u$ as a limit of
$a^m(x,D)u^m$.
This would be an advantage in as much as the expression $a^m(x,D)u^m$ is a
natural point of departure for Littlewood--Paley analysis of $a(x,D)u$ (as
explained later, cf \eqref{au_bilin-eq});
it would also make $a$ and $u$ enter in a more symmetric fashion. 
But as a drawback the resulting definition of $a(x,D)$ would then have two
steps, the first one being an
extension to $\cal F^{-1}\cal E'$ as in Section~\ref{loc-ssect}.
In comparison the limit in \eqref{ammu-eq} only refers to $S^{-\infty}$, 
cf \eqref{ammu'-eq}, which made it possible to state
Definition~\ref{gdef-defn} directly; cf \eqref{stabl-id}.

Formula \eqref{alter-eq} is so self-suggesting that it is convenient 
to write $a^m(x,D)u^m$ without further explanation,
instead of the slightly tedious 
$\OP(a^m(1\otimes\psi_m))u$, that enters Definition~\ref{gdef-defn}. 
(This is permitted as the two expressions are equal for every choice 
of the auxiliary function $\psi$, cf Lemma~\ref{alter-lem}).

Since Definition~\ref{gdef-defn} is based on a limit of $a^m(x,D)u^m$, it is
useful to relate the distribution kernel $K_m(x,y)$ 
of $u\mapsto a^m(x,D)u^m$ to the kernel $K(x,y)$ of $a(x,D)$.

The symbol of $a^m(x,D)u^m$ is $a^m(1\otimes\psi_m)\in S^{-\infty}$, cf
\eqref{alter-eq}, so \eqref{Kxy-eq} and the definition of $a^m$ give,
for all $u$, $v\in \cal S(\Rn)$,
\begin{equation}
  \begin{split}
  \dual{a^m(x,D)u^m}{v}&=
  \dual{\cal F^{-1}_{\eta\to y}(a^m(1\otimes\psi_m))(x,x-y)}{v(x)\otimes u(y)}
\\
  &=
  \dual{\cal F^{-1}(\psi_m(\xi)\psi_m(\eta)
             \cal F_{x\to\xi} a(\xi,\eta)))\circ M}{v\otimes u}.
  \end{split}
  \label{amumv-eq}
\end{equation}
Because
$\cal F\cal F^{-1}_{\eta\to y}\cal F^{-1}_{\xi\to x}=I$ on $\R^{2n}$,
and 
$M=\left(\begin{smallmatrix}I&0\\I&-I\end{smallmatrix}\right)=M^{-1}$,
formula \eqref{amumv-eq} shows that
\begin{equation}
  K_m(x,x-y)=\cal F^{-1}((\psi_m\otimes\psi_m)\cal F(K\circ M))(x,y).
  \label{Kmxxy-eq}
\end{equation}
This can be restated as follows:

\begin{prop}   \label{Kreg-prop}
When $a\in S^{\infty}_{1,1}$ and $\psi\in C^\infty_0(\Rn)$ equals $1$ in a
neighbourhood of the origin, then the distribution kernel $K_m(x,y)$ of 
$u\mapsto a^m(x,D)u^m$, cf \eqref{alter-eq}, 
is the function in $C^\infty(\Rn\times\Rn)$ given by 
\begin{equation}
  K_m(x,y)=\cal F^{-1}(\psi_m\otimes\psi_m)*(K\circ M) (x,x-y),
  \label{Kmxy-eq}
\end{equation}
which is the conjugation by $\circ M$ of the convolution
of $K(x,y)$ by 
$4^{nm}\cal F^{-1}\psi(2^mx)\cal F^{-1}\psi(2^my)$.
\end{prop}

Naturally, this result will be useful for the discussion in the next section.

\section{Preservation of $C^\infty$-smoothness}   \label{pslp-sect}
It is well known that every classical pseudo-differential operator 
$A=a(x,D)$  is pseudo-local, 
\begin{equation}
  \singsupp Au\subset \singsupp u\quad\text{for every}\quad u\in D(A).
  \label{psdloc-eq}
\end{equation}
In the context of type $1,1$-operators, the requirement $u\in D(A)$ should be
made explicitly as the domain $D(A)$ in many cases will be only a proper
subspace of $\cal S'(\Rn)$. 

It could be useful to call $\Omega:=\Rn\setminus\singsupp u$ the
\emph{regular set} of $u$, for this set has the important property that
regularisations of $u$ converge (not just in $\cal S'(\Rn)$ but also)
in the topology of $C^\infty(\Omega)$. This fact could well be 
folklore, but references seem unavailable, and since it is the crux of
the below proof of pseudo-locality, details are given for the reader's
convenience.

\begin{lem}[The regular convergence lemma]
  \label{regconv-lem}
Let $u\in \cal S'(\Rn)$ and set $\psi_k(\xi)=\psi(\varepsilon_k\xi)$ for some
sequence $\varepsilon_k\searrow 0$ and $\psi\in \cal S(\Rn)$. Then 
\begin{equation}
 \psi_k(D) u\to \psi(0)\cdot u
  \quad\text{for}\quad k\to\infty
  \label{psiku-eq}
\end{equation}
in the Fr\'echet space $C^\infty(\Rn\setminus\singsupp u)$. 
If $\cal F^{-1}\psi\in C^\infty_0(\Rn)$ the conclusion holds 
for all $u\in \cal D'(\Rn)$, 
if $\psi_k(D)u$ is replaced by $(\cal F^{-1}\psi_k)*u$. 
\end{lem}

In the topology of $\cal S'(\Rn)$ the well-known property 
\eqref{psiku-eq} is easy, for test against $\bar\varphi\in \cal S(\Rn)$
reduces \eqref{psiku-eq} to the fact that $\psi(\varepsilon_k\cdot )\hat
\varphi\to \psi(0)\hat \varphi$ in $\cal S(\Rn)$.

The main case is of course $\psi(0)=1$.
For $\psi(0)=0$ one obtains the occasionally useful
fact that $\psi_k(D)u\to 0$ in
$C^\infty$ over the regular set of $u$.

\begin{proof}
Let $K\Subset\Rn\setminus\singsupp u=:\Omega$ 
and take a partition of unity 
$1=\varphi+\chi$  with $\varphi\in C^\infty(\Rn)$ 
such that $\varphi\equiv1$ on a neighbourhood of $K$ 
and $\supp\varphi\Subset\Omega$. This gives a splitting
$\psi_k(D)u=\psi_k(D)(\varphi u)+\psi_k(D)(\chi u)$ where
$\varphi u\in C^\infty_0(\Rn)$.
Since $\int\cal F^{-1}\psi\,dx=\psi(0)$,
\begin{multline}
  D^\alpha[\check\psi_k*(\varphi u)-\psi(0)\varphi u]
  = \int \check\psi(y)[D^\alpha(\varphi u)(x-\varepsilon_k y)-
                        D^\alpha(\varphi u)(x)]\,dy
\\
  = -\sum_{j=1}^n\int\int_{0}^1 \check\psi(y)\partial_{x_j} 
    D^\alpha(\varphi u)(x-\theta\varepsilon_k y)
\varepsilon_k y_j\,d\theta dy.
\end{multline}
Using the seminorms in \eqref{Snorm-eq} in a crude way,
\begin{equation}
  |\partial_{j} D^\alpha(\varphi u)(x-\theta\varepsilon_k y)
   \varepsilon_k y_j|\le
  \varepsilon_k |y|\norm{\varphi u}{\cal S, |\alpha|+1}
  \searrow 0 ,
  \label{nDau-eq}
\end{equation}
so consequently $D^\alpha(\check\psi_k*(\varphi u))\to \psi(0)D^\alpha u$
uniformly on $K$.

For $\psi_k(D)(\chi u)$ it is used that continuity of 
$\cal S(\Rn)\xrightarrow{u}\C$ gives $c$, $N>0$ such that
\begin{equation}
  |\dual{\chi u}{D^\alpha_x\check\psi_k(x-\cdot)}|  
  \le c\sup_{y\in \Rn,\,|\beta|\le N}
        \frac{\ang{y}^N}{\varepsilon_k^{n+|\alpha|}}
        \bigl| D^{\beta}_y
        (\chi(y)
        D^\alpha\check\psi(\frac{x-y}{\varepsilon_k})) \bigr|.
  \label{uconst-ineq}
\end{equation}
Here $0<d:=\op{dist}(K, \supp \chi)\le |x-y|$ 
for $x\in K$, $ y\in \supp\chi$, so every negative power of $\varepsilon_k$
fulfils 
$\varepsilon_k^{-l}\le(1+\varepsilon_k^{-1}|x-y|)^l/d^l$.
Moreover, $(1+|y|)^N\le c_K(1+|x-y|/\varepsilon_k)^N$ for
$\varepsilon_k<1$. So evaluation of an $\cal S$-seminorm at $\cal F^{-1}\psi$
yields $\sup_K |D^\alpha(\cal F^{-1}\psi_k*(\chi u))|\le
C\varepsilon_k\searrow 0$. 

When $\cal F^{-1}\psi\in C^\infty_0$, clearly
$(\cal F^{-1}\psi_k)*(\chi u)=0$ around $K$ eventually.
\end{proof}

In the sequel the main case is the one in which $\psi$ itself
has compact support, so the proof above is needed. 

\subsection{The pseudo-local property}
The following sharpening of Lemma~\ref{regconv-lem} shows that, 
in certain situations, one even
has convergence $f\psi_k(D)u\to fu$ in $\cal S$.
To obtain this in a general set-up, let $x\in \Rn$ be split in two groups
as $x=(x',x'')$ with $x'\in \R^{n'}$, $x''\in \R^{n''}$.

\begin{prop}   \label{Sconv-prop}
Suppose $u\in \cal S'(\Rn)$ has $\singsupp u\subset 
\{\,x=(x',x'')\mid x''=0\,\}$ and that $f u\in \cal S(\Rn)$
for every $f$ in the subclass 
$\cal C\subset C^\infty_{\b}(\Rn)$ consisting of the $f$ for which $|x'|$ is
bounded on $\supp f$ and $\supp f\cap \{\,x\mid x''=0\,\}=\emptyset$.
Then 
\begin{equation}
  f \cal F^{-1}(\psi_k\hat u)
  \xrightarrow[k\to\infty]{}
  \psi(0)fu
  \quad\text{in $\cal S(\Rn)$},
\end{equation}
for every sequence $\psi_k=\psi(\varepsilon_k\cdot)$ 
given as in Lemma~\ref{regconv-lem}.
\end{prop}

\begin{proof}
For $f\in \cal C$ it is straightforward to see that there is a $\delta$ such
that
\begin{equation}
  \inf\{\,|x''|\mid x\in \supp f \,\}\ge \delta>0.
\end{equation}
One can then take $\varphi\in \cal C$ such that
$\varphi\equiv1$ where $|x''|\ge\delta/2$, hence on $K=\supp f$.
Mimicking the proof of Lemma~\ref{regconv-lem}, 
compactness of $K$ is not needed since 
$\varphi u$ is in $\cal S$ by assumption. 
Instead of \eqref{nDau-eq} one should estimate
\begin{equation}
  \ang{x}^N|\partial_j D^\alpha (\varphi u)(x-\theta\varepsilon_k y)
   \varepsilon_ky_j|,
\end{equation}
but $(1+|x|)^N\le(1+|x-\theta\varepsilon_k y|)^N(1+|y|)^N$ when
$\varepsilon_k<1$, so it follows mutatis mutandis that 
for an arbitrary seminorm,
\begin{equation}
  \norm{f\check\psi_k*(\varphi u)
        -\psi(0)f u}{\cal S, N}\le 
  c\varepsilon_k\searrow 0.
\end{equation}
And because $\chi=1-\varphi$ fulfils
$d=\op{dist}(K,\supp \chi)\ge\tfrac{\delta}{2}>0$, one gets as in
\eqref{uconst-ineq},
\begin{equation}
  \norm{f\check\psi_k*(\chi u)}{\cal S, N}\le c\varepsilon_k  .
\end{equation}
Indeed, 
$\ang{x}^N\le (1+|y|)^N(1+|x-y|/\varepsilon_k)^N$ and now factors like
$\ang{y}^N$ are harmless as
\begin{equation}
  (1+|y|)^N\le (1+|y'|)^N(1+|y''|)^N
  \le (1+|x'|)^N(1+|y-x|)^N(1+|y''|)^N,
\end{equation}
where $|y''|<\delta$ on $\supp\chi$ whilst $|x'|$ is bounded on $\supp f$.
\end{proof}

For distribution kernels there is a similar result, but in this case
it is well known that one need not assume rapid decay: 
let $f\in C^\infty_{\b}(\R^{2n})$ 
have its support disjoint from the diagonal 
$\Delta=\{\,(x,x)\mid x\in \Rn\,\}$ and bounded in
the $x$-direction, that is
\begin{equation}
\begin{gathered}
  \Delta\cap\supp f=\emptyset
  \\
  \exists R>0\colon (x,y)\in \supp f\implies |x|\le R.
\end{gathered}
  \label{fsupp-cnd}  
\end{equation}
Then $f(x,y)K(x,y)$ is in $\cal S(\R^{2n})$ whenever $K$ is the kernel of a
type $1,1$-operator. Indeed,
\begin{equation}
  (1+|(x,y)|)^N\le (1+|x|)^N(1+|y|)^N
  \le (1+|x|)^{2N}(1+|y-x|)^N
\end{equation}
and here $|x|$ is bounded on $\supp f$, so by setting $z=x-y$ in \eqref{zK-eq}
one has that $\ang{(x,y)}^N D^\alpha_xD^\beta_y (fK)$ is bounded for all
$N$, $\alpha$, $\beta$. 
Invoking Proposition~\ref{Sconv-prop} this gives

\begin{prop}
  \label{SKconv-prop}
If $a\in S^\infty_{1,1}(\Rn\times\Rn)$ has kernel $K$ and
$K_m$ is the approximating kernel given by \eqref{Kmxy-eq}, then 
it holds for every $f\in C^\infty_{\b}(\R^{2n})$ with the property 
\eqref{fsupp-cnd} that
\begin{equation}
  fK_m \xrightarrow[m\to\infty]{}
  fK   \quad\text{in}\quad \cal S(\R^{2n}).
\end{equation}
\end{prop}
\begin{proof}
The class $\cal C$ of Proposition~\ref{Sconv-prop} contains $f(x,x-y)$, 
and Proposition~\ref{Kreg-prop} gives
\begin{equation}
  f(x,x-y)K_m(x,x-y)=f(x,x-y)
   \cal F^{-1}(\psi_m\otimes\psi_m)*(K\circ M).
\end{equation}
The right-hand side tends to $(fK)\circ M$ in $\cal S(\R^{2n})$ according to
Proposition~\ref{Sconv-prop}, so it remains to use the continuity of $\circ
M$ in $\cal S(\Rn)$.
\end{proof}

It can now be proved that operators of type
$1,1$ are pseudo-local.
The argument below is classical up to
the appeal to \eqref{uphiK-eq}. In case $A$ is $\cal
S'$-continuous, this formula follows at once from the
density of $\cal S$ in $\cal S'$. However, in general $A$ is not even
closable, but instead the limiting procedure of
Definition~\ref{gdef-defn} applies via the approximation 
in Proposition~\ref{SKconv-prop}.

\begin{thm}
  \label{psdloc-thm}
For every $a\in S^\infty_{1,1}(\Rn\times\Rn)$ the operator $A=a(x,D)$ has the pseudo-local property; that is
$\singsupp Au \subset \singsupp u$ for every $u\in D(A)$.
\end{thm}
\begin{proof}
Let $\psi,\chi\in C^\infty_0(\Rn)$ have supports disjoint from
$\singsupp u$ such that $\chi\equiv1$ around $\supp\psi$. 
Then $\chi u\in C^\infty_0(\Rn)$ so that also $(1-\chi)u$ is in
the subspace $D(A)$ and
\begin{equation}
  \psi Au=\psi A(\chi u)+\psi A(1-\chi )u.
\end{equation}
Here  $\psi A(\chi u)\in C_0^{\infty}(\Rn)$ since $A\colon \cal S\to\cal S$,
while $\psi A(1-\chi )u$ is seen at once to have kernel
\begin{equation}
  \tilde K(x,y)=\psi(x)K(x,y)(1-\chi(y)).
\end{equation}
The function 
$f(x,y)=\psi(x)(1-\chi(y))$ fulfils \eqref{fsupp-cnd},
for $\Delta$ contains no contact point of $\{\,f\ne0\,\}$ because
$\op{dist}(\supp\psi,\supp(1-\chi))>0$. Therefore $\tilde K\in\cal S(\R^{2n})$
as seen after \eqref{fsupp-cnd}.
This strongly suggests that, with $\chi_1=1-\chi$,
\begin{equation}
  \dual{\psi A\chi_1 u}{\varphi}=\dual{\varphi\otimes u}{\tilde K}
  \quad\text{for all $\varphi\in C^\infty_0(\Rn)$}.
  \label{uphiK-eq}
\end{equation}
And it suffices to prove this identity, for by definition of the tensor
product it entails that 
$\psi A\chi_1 u=\dual{u}{\tilde K(x,\cdot)}$
which is a $C^\infty$-function of $x\in \Rn$.

Now if $A_m:=\OP(a^m(x,\eta)\psi_m(\eta))$ and $K_m$ is its kernel, 
one can take 
$u_l\in C^\infty_0(\Rn)$ such that $u_l\to u$ in $\cal S'$. 
Applying Definition~\ref{gdef-defn} to $A$,
the $\cal S'$-continuity of $A_m$ gives
\begin{equation}
  \dual{\psi A\chi_1 u}{\varphi}=   
  \lim_{m\to\infty}\dual{ A_m\chi_1 u}{\psi\varphi}=
  \lim_{m\to\infty}\lim_{l\to\infty}
     \dual{K_m}{(\psi\varphi)\otimes(\chi_1u_l)}. 
\end{equation}
Here $K_m\in C^\infty(\R^{2n})$ by Lemma~\ref{Kreg-prop}, 
so for the right-hand side one finds,
since $u\mapsto \varphi\otimes u$ is $\cal S'$-continuous 
and $fK_m\in \cal S(\R^{2n})$,
\begin{equation}
\int {\psi(x)\varphi(x)\chi_1(y)u_l(y)}{K_m(x,y)}\,d(x,y)
\xrightarrow[l\to\infty]{}
   \dual{\varphi\otimes u}{(\psi\otimes\chi_1)K_m}.
\end{equation}
As $(\psi\otimes \chi_1)K_m=fK_m\to fK=\tilde K$
in $\cal S(\R^{2n})$ by Proposition~\ref{SKconv-prop}, the proof is
complete. 
\end{proof}

\begin{rem}   \label{psdloc-rem}
Theorem~\ref{psdloc-thm} was anticipated by C.~Parenti and
L.~Rodino \cite{PaRo78}, although they just appealed to the fact that
$K(x,y)$ is $C^\infty $ for $x\ne y$. This does not quite suffice as $\psi
A\chi_1u$ should be identified with a $C^\infty $-function, 
eg $\dual{u}{\tilde K(x,\cdot )}$,
for $u\in D(A)\setminus \cal S(\Rn)$; which is
non-trivial in the absence of continuity and the usual rules of calculus.
\end{rem}

\subsection{A digression on products}
The opportunity is taken here to settle an open problem for the generalised
pointwise product $\pi(u,v)$ mentioned in Remark~\ref{prod-rem}.
 
\bigskip

First the commutation of pointwise
multiplication and vanishing frequency modulation is discussed.
Let $u\in \cal S'(\Rn)$ and $f\in \cal O_M(\Rn)$ be given and
$\psi_m=\psi(2^{-m}\cdot)$ for some arbitrary $\psi\in \cal S(\Rn)$ with
$\psi(0)=1$. Approximating $fu$ in two ways in $\cal S'$,
\begin{equation}
  B_mu:=\psi_m(D)(fu)- f\psi_m(D)u\to 0\quad\text{for}\quad m\to\infty.
  \label{fucomm-eq}
\end{equation}
This commutation in the limit is not, however, a direct consequence of
pseudo-differential calculus, for the commutator $B_m$ has amplitude
$b_m(x,y,\eta)=(f(y)-f(x))\psi(2^{-m}\eta)$,
which is in the space of symbols with
estimates 
$|D^\alpha_\eta D^\beta_{x,y}a(x,y,\eta)|\le
C_{\alpha,\beta,K,N}\ang{\eta}^{-N}$ for all $N>0$, $K\Subset\Rn\times\Rn$.
As such $\OP(b_m(x,y,\eta))$ is only defined on $\cal E'(\Rn)$.

However, 
\eqref{fucomm-eq} is seen at once to 
hold in $C^\infty(\Rn\setminus\singsupp u)$,
by using Lemma~\ref{regconv-lem} on both terms. 
The next results confirms that $B_mu\to0$ even in $C^\infty (\Rn)$,
despite the singularities of $u$. The idea is to use 
Lemma~\ref{regconv-lem} once more to get 
a reduction to $f\in C^\infty_0(\Rn)$, so that $B_m\to 0$ in the globally
estimated class $\OP(S^{-\infty}(\Rn\times\Rn))$:

\begin{prop}   \label{comm-prop}
When $u\in \cal S'(\Rn)$, $f\in \cal O_M(\Rn)$,
and $\psi\in \cal S(\Rn)$ with $\psi(0)=1$, it holds true that
$\lim_{m\to\infty}( \psi_m(D)(fu)-f\psi_m(D)u)=0$ in the topology of 
$C^{\infty}(\Rn)$.
\end{prop}
\begin{proof}
When $\chi\in C^\infty_0(\Rn)$ equals $1$ on a neighbourhood of a given
compact set $K\subset\Rn$, then $K$ is contained in the regular set of 
$(1-\chi) u$, so it follows as above from Lemma~\ref{regconv-lem} that 
$\sup_{K,|\alpha|\le l}|D^\alpha(\psi_m(D)(f(1-\chi) u)-
f\psi_m(D)(1-\chi) u)|\to 0$ for $m\to\infty$. 

It now suffices to cover the case in which 
$K\subset\supp u\subset\supp f\Subset\Rn$. 
Then $B_m$ has symbol
\begin{equation}
  b_m(x,\eta)=(e^{\im D_x\cdot D_\eta}-1)f(x)\psi_m(\eta)\in 
  S^{-\infty}(\Rn\times\Rn).
\end{equation}
However, 
$u\in  H^t$ for some $t<0$, and $B_m\in \BBb B(H^t,H^s)$ for all $s>0$,
whence 
\begin{equation}
  \sum_{|\alpha|\le l} \nrm{D^\alpha B_m u}{\infty}
  \le c\nrm{B_mu}{H^s}\le c\nrm{B_m}{}\nrm{u}{H^t} \quad\text{for $s>l+n/p$}.
\end{equation}
It remains to show that the operator norm $\nrm{B_m}{}\to0$.
Using direct estimates as in eg \cite[Prop.~2.2]{H88}, it is enough to show
for all $N>0$, $\alpha$, $\beta$ that
\begin{equation}
  \lim_{m\to\infty}\sup_{\eta\in \Rn}\ang{\eta}^N
   |D^\beta_x D^\alpha_\eta b_m(x,\eta)|\to 0\quad\text{for}\quad m\to\infty.
\end{equation}
But $D^\alpha_\eta$, $D^\beta_{x}$ commute with $e^{\im D_x\cdot D_{\eta}}$,
so it suffices to  treat $\alpha=0=\beta$ for general $f$ and $\psi$, ie to
show that uniformly in $x\in \Rn$
\begin{equation}
  |(e^{\im D_x\cdot D_\eta}-1)f(x)\psi_m(\eta)|\le c
  \sum_{|\alpha|+|\beta|\le 2n+2}|D^\beta_{x}D^\alpha_{\eta}
    D_x\cdot D_\eta f(x)\psi_m(\eta)|
  =\cal O(2^{-m(|\alpha|+1)}\ang{\eta}^{-N}).
\end{equation}
The estimate to the left is known, and follows directly from
\cite[Prop.~B.2]{H88}. 

Altogether $\sup_{x\in  K, |\alpha|\le l}|D^\alpha B_mu|\to0$ for
$m\to\infty$, as claimed.
\end{proof}

Besides being of interest in its own right, 
Proposition~\ref{comm-prop} gives at once a natural property of associativity 
for the product
$\pi$ in Remark~\ref{prod-rem}.

\begin{thm}   \label{pi-thm}
The product $(u,v)\mapsto\pi(u,v)$ is partially associative, ie
when $(u,v)\in \cal S'(\Rn)\times\cal S'(\Rn)$ 
is in the domain of $\pi(\cdot,\cdot)$ so is
$(fu,v)$ and $(u,fv)$ for every $f\in \cal O_M(\Rn)$ and
\begin{equation}
  f\pi(u,v)=\pi(fu,v)=\pi(u,fv).
\end{equation}
\end{thm}
\begin{proof}
For every $\varphi\in C^\infty(K)=\bigl\{\,\psi\in C^\infty_0(\Rn) \bigm|
\supp\psi\subset K \,\bigr\}$, $K\Subset\Rn$\,
\begin{equation}
  \dual{(fu)^mv^m}{\varphi}-\dual{f\cdot u^mv^m}{\varphi}
  = \dual{v^m}{((fu)^m-f\cdot u^m)\varphi}
   \xrightarrow[m\to\infty]{~} 0,
\end{equation}
for by Banach--Steinhauss' theorem it suffices that
$((fu)^m-f\cdot u^m)\varphi\to0$ in $C^\infty(K)$, which holds since
$(fu)^m-f\cdot u^m\to0$ in $C^\infty(\Rn)$ according to
Proposition~\ref{comm-prop}. Hence $f\pi(u,v)=\pi(fu,v)$; the other identity
is justified similarly.
\end{proof}

\section{Extended action of distributions}
  \label{exdist-sect}
To prepare for Section~\ref{supp-sect} it is exploited that the map
$(u,f)\mapsto\dual{u}{f}$ is defined also for certain 
$u$, $f$ in $\cal D'(\Rn)$ that do not belong to dual spaces. This
bilinear form is moreover shown to have a property of stability under regular
convergence. 

\bigskip

\subsection{A review}
First it is recalled that the product $fu$ is defined for 
$f,u\in \cal D'(\Rn)$ if
\begin{equation}
  \singsupp f\bigcap\singsupp u=\emptyset.
  \label{ssfu-eq}
\end{equation}
In fact, $\Rn$ is covered by $Y_1=\Rn\setminus\singsupp u$ 
and $Y_2=\Rn\setminus\singsupp f$; in $Y_1$ there is a product
$(fu)_{Y_1}\in \cal D'(Y_1)$
given by $\dual{(fu)_{Y_1}}{\varphi}=\dual{f}{u\varphi}$ for $\varphi\in
C^\infty_0(Y_1)$, and similarly $\dual{u}{f\varphi}$, $\varphi\in
C^\infty_0(Y_2)$ defines a product $(fu)_{Y_2}\in \cal D'(Y_2)$; and for
$\varphi\in C^\infty_0(Y_1\cap Y_2)$ both products are given
by the $C^\infty$-function $f(x)u(x)$ so they coincide on $Y_1\cap Y_2$;
hence $fu$ is well defined
in $\cal D'(\Rn)$ and given on $\varphi\in C^\infty_0(\Rn)$ by the following
expression, where the splitting $\varphi=\varphi_1+\varphi_2$ for
$\varphi_j\in C^\infty_0(Y_j)$ is obtained from a partition of unity,
\begin{equation}
  \dual{fu}{\varphi}=\dual{f}{u\varphi_1}+\dual{u}{f\varphi_2}.
  \label{fu-id}
\end{equation}
This follows from the \emph{recollement de morceaux} theorem,
cf \cite[Thm.~I.IV]{Swz66} or \cite[Thm.~2.2.4]{H}; 
by the proof of this, \eqref{fu-id} 
does not depend on how the partition is chosen.

\begin{rem}  \label{split-rem}
Therefore, when $F_1$, $F_2$ are 
closed sets in $\Rn$ given with the properties 
$\singsupp u\subset F_1$, $\singsupp f\subset F_2$
and $F_1\cap F_2=\emptyset$ (so that $\Rn$ is covered by their complements)
one can always take the splitting in \eqref{fu-id} such that
$\varphi_1\in C^\infty_0(\Rn\setminus F_1)$,
$\varphi_2\in C^\infty_0(\Rn\setminus F_2)$.
\end{rem}

Secondly $f\mapsto \dual{u}{f}$ for $u\in \cal D'(\Rn)$
is a well defined linear map on the subspace of $f\in \cal D'(\Rn)$ such 
that \eqref{ssfu-eq} holds together with
\begin{equation}
  \supp u\cap \supp f\Subset\Rn.
  \label{suppfu-eq}
\end{equation}
In fact $\dual{u}{f}:=\dual{fu}{1}$ is possible: $fu$ is defined by
\eqref{ssfu-eq} and is in $\cal E'$ by \eqref{suppfu-eq},
so by \cite[Th~2.2.5]{H} the map
$\psi\mapsto\dual{fu}{\psi}$
extends from $C^\infty_0(\Rn)$
to all $\psi\in C^\infty(\Rn)$,   
uniquely among the extensions that vanish when 
$\supp \psi\cap \supp fu=\emptyset$; hence it is defined on  
the canonical choice $\psi\equiv1$, and for all $\varphi\in
C^\infty_0(\Rn)$ equal to $1$ around $\supp fu$,
\begin{equation}
  \dual{u}{f}=\dual{fu}{1}=\dual{fu}{\varphi}.
\end{equation}

These constructions have been quoted in a slightly modified form from
\cite[Sect.~3.1]{H}. 
The definition implies that $(f,u)\mapsto fu$ is bilinear; it
is clearly commutative and is partially associative in the sense that
$\psi(fu)=(\psi f)u=f(\psi u)$ when $\psi\in  C^\infty(\Rn)$ while $f$, $u$
fulfill \eqref{ssfu-eq}. This also yields
\begin{equation}
  \supp fu\subset \supp f \cap \supp u.
\end{equation}

When applying cut-off functions,
partial associativity entails $(\chi f)(\varphi u)=fu$ when 
$\chi$, $\varphi$ equal $1$ around $\supp f\cap\supp u$.
Therefore  test against $1$ gives that
$\dual{\varphi u}{\chi f}=\dual{u}{f}$.

\subsection{Stability under regular convergence}
The product $fu$ is not continuous, for $f=0$ is the limit in $\cal D'$ of
$f^\nu=e^{-\nu |x|^2}\in C^\infty $ and for $u=\delta_0$ it is clear that
$f^\nu u=\delta_0\not\to 0=fu$.
As a remedy it is noted that
$fu$ is separately stable under regular convergence; cf
Lemma~\ref{regconv-lem}. This carries over to the
extended bilinear form $\dual{\cdot }{\cdot }$ 
under a compactness condition:

\begin{thm}   \label{fuconv-thm}
Let $u$, $f\in \cal D'(\Rn)$, $f^\nu\in C^\infty(\Rn)$ fulfil
$\lim_{\nu}f^\nu= f$ in $\cal D'(\Rn)$ and in
$C^\infty(\Rn\setminus F)$ for $F=\singsupp f$.
When $u$, $f$ have disjoint singular supports, cf \eqref{ssfu-eq}, then
\begin{equation}
  f^\nu u\to fu \quad\text{in $\cal D'(\Rn)$ for}\quad \nu\to\infty.
  \label{fnuu-eq}
\end{equation}
If moreover $\supp u\bigcap\supp f$ is compact
and $\chi\in C^\infty_0(\Rn)$ equals $1$ around this set, then
\begin{equation}
  \lim_{\nu\to\infty}\dual{\chi u}{f^\nu}
 =\lim_{\nu\to\infty}\dual{u}{\chi f^\nu}=\dual{u}{f}.
  \label{ufnu-eq}
\end{equation} 
Here one can take $\chi\equiv 1$ on $\Rn$ if a compact set contains 
$\supp u$ or $\bigcup_{\nu}\supp(f^\nu u)$. 
The conclusions hold verbatim when $F\subset \Rn$ is closed 
and $\singsupp f\subset F\subset (\Rn\setminus \singsupp u)$.
\end{thm}
\begin{proof}
To show \eqref{fnuu-eq} for a general $F$,
note that \eqref{fu-id} applies to the product $f^\nu u$ of $f^\nu\in
C^\infty$ and $u\in \cal D'$. Using Remark~\ref{split-rem} and that
$f^\nu\to f$ in $C^\infty (\Rn\setminus F)$,
one has $f^\nu\varphi_2\to f\varphi_2$ in $C^\infty_0(\Rn\setminus F)$; 
the other term on the
right-hand side of \eqref{fu-id} converges by the $\cal D'$-convergence of
the $f^\nu$. 
Therefore $\dual{f^\nu u}{\varphi}\to
\dual{f}{u\varphi_1}+\dual{u}{f\varphi_2}=\dual{fu}{\varphi}$.

By the definition of $\dual{u}{f}$
above, when $\chi$ is as in the theorem, then the
just proved fact that $f^\nu u\to fu$ in $\cal D'$ leads to
\eqref{ufnu-eq} since
\begin{equation}
  \dual{u}{f}=\dual{fu}{1}=\dual{fu}{\chi}
  =\lim_{\nu\to\infty}\dual{f^\nu u}{\chi}.
  \label{ufnuchi-eq}
\end{equation}
When $\bigcup_{\nu}\supp(f^\nu u)$ is precompact and $\chi=1$ on a
neighbourhood, then $0=\dual{f^\nu u}{1-\chi}$
can be added to \eqref{ufnuchi-eq}, 
which yields $\lim_{\nu}\dual{u}{f^\nu}$ by the extended
definition of $\dual{\cdot}{\cdot}$.
\end{proof}

\begin{rem}   \label{cutoff-rem}
In general \eqref{ufnu-eq} cannot hold without the cut-off function $\chi$.
Eg for $n\ge2$ and $x=(x',x_n)$ one may take
$f=1_{\{x_n\le0\}}$ and $u=1_{\{x_n\ge 1/|x'|^2\}}$, so that
$\dual{u}{f}=0$. Setting $f^\nu=2^{n\nu}\varphi(2^\nu\cdot)*f$ 
for $\varphi\in C^\infty_0$ with $\varphi\ge0$, $\int\varphi=1$,
and
$\supp\varphi\subset\{\,(y',y_n)\mid 1\le y_n\le 2,\ |y'|\le1\,\}$, 
it holds for $x\in \Sigma_\nu=\{\,0\le x_n\le 2^{-\nu} \,\}$ that 
$x_n-2^{-\nu}y_n\le0$ on $\supp\varphi$ so that
\begin{equation}
  f^\nu(x)=\int \varphi(y) f(x-2^{-\nu}y)\, dy  =\int \varphi\,dy=1.
\end{equation}
Hence $\supp u\cap\supp f^\nu$ is unbounded,
so $\dual{u}{f^\nu}$ is undefined (hardly just a technical
obstacle as $\dual{u}{f^\nu}=\int uf^\nu \,dx=\infty$ would be the value). 
\end{rem}

\subsection{Consequences for kernels}
Although it is on the borderline of the present subject, it would not be
natural to omit that Theorem~\ref{fuconv-thm} gives an easy way to extend
the link between an  operator and its kernel:

\begin{thm}   \label{AK-thm}
Let $A\colon\cal S'(\Rn)\to\cal S'(\Rn)$ be a  continuous linear map
with distribution kernel $K(x,y)\in \cal S'(\Rn\times\Rn)$.
Suppose that $u\in \cal S'(\Rn)$ and $v\in C^\infty_0(\Rn)$ satisfy
\begin{gather}
  \supp K\bigcap \supp  v\otimes u \Subset\Rn\times\Rn,
  \label{suppKuv-eq}
  \\
  \singsupp K\bigcap \singsupp  v\otimes u=\emptyset.
  \label{ssuppKuv-eq}
\end{gather}
Then $\dual{Au}{ v}=\dual{K}{v\otimes u}$, 
with extended action of $\dual{\cdot}{\cdot}$.
When $A$ is a continuous linear map
$\cal D'(\Rn)\to\cal D'(\Rn)$, this is valid for 
$u\in \cal D'(\Rn)$, $v\in C^\infty_0(\Rn)$ fulfilling
\eqref{suppKuv-eq}--\eqref{ssuppKuv-eq}. 
\end{thm}

\begin{proof}
By the conditions on $u$ and $v$, the expression $\dual{K}{v\otimes u}$ is
well defined. By mollification there is regular convergence to 
$u$ of a sequence $u_\nu\in C^\infty(\Rn)$; this gives
\begin{equation}
  v\otimes u_\nu\xrightarrow[\nu\to\infty]{~} v\otimes u
  \quad\text{in}\quad \cal S'(\Rn\times\Rn) \quad\text{and}\quad
C^\infty(\Omega) 
\end{equation}
when $\Omega=(\Rn\times\Rn)\setminus(\supp v\times \singsupp u)
=\R^{2n}\setminus\singsupp(v\otimes u)$. 
Applying Theorem~\ref{fuconv-thm} on $\R^{2n}$, the cut-off function  may
be taken as 
$\kappa(x)\chi(y)$ for some $\kappa,\chi\in C^\infty_0(\Rn)$ such that
$\kappa$ equals $1$ on $\supp v$ and $\kappa\otimes\chi=1$ on the compact
set $\supp K\cap\supp(v\otimes u)$. This gives
\begin{equation}
  \dual{K}{v\otimes u}
  =\lim_{\nu\to\infty} \dual{(\kappa\otimes\chi)K}{v\otimes u_\nu}
  =\lim_{\nu\to\infty} \dual{A(\chi u_\nu)}{v}=\dual{A(\chi u)}{v}.
\end{equation}
For $\chi=\psi(2^{-m}\cdot)$ and $\psi=1$ near $0$, the conclusion
follows from the continuity of $A$ since $\psi(2^{-m}\cdot)u\to u$ in
$\cal S'$.
The $\cal D'$-case is similar.
\end{proof}

\begin{rem}
The conditions \eqref{suppKuv-eq}--\eqref{ssuppKuv-eq} are far from
optimal, for $(v\otimes u)K$ acts on $1$ if it is just an integrable
distribution, that is if $(v\otimes u)K$ belongs to $\cal D'_{L^1}=
\bigcup_m W^{-m}_1$ on $\R^{2n}$. Similarly \eqref{ssuppKuv-eq} is not
necessary for $(v\otimes u)\cdot K$ to make sense; eg it suffices that
$(x,\xi)\notin\op{WF}(K)\cap (-\op{WF}(v\otimes u))$ whenever $(x,\xi)\in
\R^{2n}$. More generally the existence of the product $\pi(K, v\otimes u)$
would suffice; cf Remark~\ref{prod-rem}.
\end{rem}

The above result applies in particular to the pseudo-differential operators
$A$ corresponding to a standard symbol space $S$, such as
$S^d_{1,0}(\Rn\times\Rn)$. So does the next consequence.

\begin{cor}
  \label{AK-cor}
When $A$ is as in Theorem~\ref{AK-thm}, it holds for every $u\in \cal S'(\Rn)$
that
\begin{equation}
  \supp Au\subset  \overline{\supp K\circ \supp u}.
  \label{Ksuppu-eq}
\end{equation}
Hereby $\supp K\circ\supp u= \bigl\{\,x\in \Rn \bigm| \exists y\in
\supp u\colon (x,y)\in \supp K \,\bigr\}$, which is a closed set if $\supp
u\Subset\Rn$. The result extends to $u\in \cal D'(\Rn)$ when $A$ is 
$\cal D'$-continuous. 
\end{cor}
\begin{proof}
Whenever $v\in C^\infty_0(\Rn)$ fulfils
$\supp v\Subset \Rn\setminus \overline{\supp K\circ\supp u}$,
then 
\begin{equation}
  \supp K\bigcap\supp( v\otimes u)=\emptyset.
  \label{Kuv-eq}
\end{equation}
For else some $(x,y)\in \supp K$ would fulfill $y\in \supp u$ and
$x\in \supp v$, in contradiction with the support condition on $v$. By
\eqref{Kuv-eq} the assumptions of Theorem~\ref{AK-thm} are satisfied,
so
$\dual{Au}{v}=\dual{K}{v\otimes u}=0$. Hence $Au=0$ holds outside the
closure of $\supp K\circ\supp u$.
\end{proof}

\begin{rem}
  \label{AK-rem}
The argument of Corollary~\ref{AK-cor} is
completely standard for $u\in C^\infty_0$, cf
\cite[Thm~5.2.4]{H} or \cite[Prop~3.1]{Shu87};
a limiting argument then implies \eqref{Ksuppu-eq} for general $u$.
However, the proof above is a direct generalisation of
the $C^\infty_0$-case, made possible by the extended action of
$\dual{\cdot}{\cdot}$ in Theorem~\ref{AK-thm}. 
This method may be interesting in its own
right; eg it extends to type
$1,1$-operators also when these are not $\cal S'$-continuous,
cf Section~\ref{supp-sect}.  
\end{rem}

\section{Kernels and transport of support}  \label{supp-sect}
Using the preceeding section, the well-known support rule is
here extended to operators of type $1,1$. As a novelty also a spectral support
rule is deduced.

\subsection{The support rule for type $1,1$-operators}
As analogues of 
Theorem~\ref{AK-thm} and Corollary~\ref{AK-cor} one has:

\begin{thm}   \label{supprule11-thm}
If $a\in S^\infty_{1,1}(\Rn\times\Rn)$ 
has kernel $K$, then 
$\dual{a(x,D)u}{v}=\dual{K}{v\otimes u}$ 
whenever $u\in D(a(x,D))$, $v\in C^\infty_0(\Rn)$ 
fulfill \eqref{suppKuv-eq}--\eqref{ssuppKuv-eq}.
And for all $u\in D(a(x,D))$ the support rule holds, ie
$\supp Au\subset \overline{\supp K\circ \supp u}$.
\end{thm}
\begin{proof}
$a(x,D)u=\lim_{m\to\infty}A_m u$ where 
$A_m=\OP(a^m(1\otimes\psi_m))\in \OP(S^{-\infty})$; its kernel $K_m$ 
is given by Proposition~\ref{Kreg-prop}. 
However, $K_m$ need not fulfil \eqref{suppKuv-eq} together with $u$, $v$,
but by use of convolutions and cut-off functions one can find
$u_\nu$ in  $C^\infty_0(\Rn)$
such that $u_\nu\to u$ in $\cal S'(\Rn)$ and in $C^\infty (\Rn\setminus
\singsupp u)$ for $\nu\to\infty $.
Then Theorem~\ref{AK-thm} gives
\begin{equation}
  \dual{Au}{v}=\lim_{m\to\infty}\lim_{\nu\to\infty }\dual{A_m u_\nu}{v}
  =\lim_{m\to\infty}\lim_{\nu\to\infty }\dual{K_m}{v\otimes u_{\nu}}.
\end{equation}
To control the supports, one can take a function $f$ 
fulfilling \eqref{fsupp-cnd} by setting $f(x,y)=g(x)h(x-y)$ for some 
$g\in C_0^\infty(\Rn)$ with $g=1$ on 
$\supp v$ and $h\in C^\infty (\Rn)$ such that 
$h(y)=0$ for $|y|<1$ while $h(y)=1$ for $|y|>2$. 
Then $K_m=fK_m+(1-f)K_m$, where the $fK_m$
tend to $fK$ in $\cal S$ according to Proposition~\ref{SKconv-prop}.
The supports of
$(1-f)K_m(v\otimes u_\nu)$, $m,\nu\in \N$, 
all lie in the precompact set $B(0,R)\times B(0,R+2)$ when
$B(0,R)\supset \supp v$, so since
$u$, $v$ are assumed to fulfil
\eqref{suppKuv-eq}--\eqref{ssuppKuv-eq}, 
Theorem~\ref{fuconv-thm} gives
\begin{equation}
  \begin{split}
  \dual{Au}{v}
  &=\lim_{m\to\infty}\lim_{\nu\to\infty }\dual{fK_m}{v\otimes u_\nu}+
   \lim_{m\to\infty}\lim_{\nu\to\infty }\dual{(1-f)K_m}{v\otimes u_\nu}
\\
  &=\dual{fK}{v\otimes u}+
    \dual{(1-f)K}{v\otimes u}
  =\dual{K}{v\otimes u}.
  \end{split}
\end{equation}
Now the support rule follows by repeating 
the proof of Corollary~\ref{AK-cor}.
\end{proof}

\subsection{The spectral support rule}
Although it has not attracted much attention, it is a natural and
useful task to
determine the frequencies entering $x\mapsto a(x,D)u(x)$.
But since
\begin{equation}
  \cal Fa(x,D)u=\cal Fa(x,D)\cal F^{-1}(\hat u)
  \label{Aconj-eq}
\end{equation}
the task is rather to control how the support of $\hat u$ is changed by
$\cal F a(x,D)\cal F^{-1}$, 
ie by the conjugation of $a(x,D)$ by the
Fourier transformation.

Even for $A\in \OP(S^{\infty}_{1,0})$ this has seemingly not been
carried out before. However, since the composite 
$\cal F A\cal F^{-1}\colon \cal
S'(\Rn)\to\cal S'(\Rn)$ is continuous for such $A$, it is straightforward to
apply Theorem~\ref{AK-thm} and Corollary~\ref{AK-cor}
to the distribution kernel 
\begin{equation}
 \cal K(\xi,\eta)=(2\pi)^{-n}\hat a(\xi-\eta,\eta)  
\end{equation}
of $\cal F A\cal F^{-1}$; cf Proposition~\ref{FKah-prop}.
This yields at once the following general result:

\begin{thm}   \label{supp10-thm}
If $a\in S^\infty _{1,0}(\Rn\times \Rn)$ and $\cal K$ is as above, then
\begin{equation}
  \supp \cal Fa(x,D)u\subset\overline{\supp\cal K\circ\supp\cal F u}
  \quad\text{for every}\quad u\in \cal S'(\Rn).
  \label{FKsuppu-eq}
\end{equation}
Here the right-hand side is closed if $\supp\cal Fu\Subset \Rn$.
\end{thm}

The result in \eqref{FKsuppu-eq}
 may also be written explicitly as in \eqref{Xi-eq}--\eqref{Xi-id}.
It is easily generalised to standard symbol spaces $S$ such as
$S^\infty _{\rho,\delta}(\Rn\times \Rn)$ with
$\delta<\rho$.
For elementary symbols in the sense of \cite{CoMe78} the spectral support rule
\eqref{FKsuppu-eq} follows at once, but as it stands
Theorem~\ref{supp10-thm}
seems to be a new result even for classical type $1,0$-operators.
The reader is referred to \cite[Sect.~1.2]{JJ05DTL} for more remarks on
Theorem~\ref{supp10-thm}, 
in particular that it makes it unnecessary to reduce to
elementary symbols in the $L_p$-theory 
(which is implicitly sketched in Section~\ref{cont-sect} below).

To extend the above to type $1,1$-operators, the next result applies to
the conjugated operator 
$\cal Fa(x,D)\cal F^{-1}$ instead of Theorem~\ref{AK-thm}.

\begin{thm}   \label{Kuv-thm}
Let $a\in S^\infty_{1,1}(\Rn\times\Rn)$ and denote
by $\cal K(\xi,\eta)=(2\pi)^{-n}\hat a(\xi-\eta,\eta)$
the distribution kernel of $\cal F a(x,D)\cal F^{-1}$; and
suppose $u\in D(a(x,D))\subset \cal S'(\Rn)$ is such that, for some $\psi$
as in Definition~\ref{gdef-defn}, 
\begin{equation}
  a(x,D)u=\lim_{m\to\infty} a^m(x,D)u^m \quad\text{holds in}\quad \cal S'(\Rn),
  \label{S'conv-eq}
\end{equation}
and that $\hat v\in C^\infty_0(\Rn)$ satisfies
\begin{equation}
  \supp\cal K\bigcap \supp \hat v\otimes\hat u \Subset\Rn\times\Rn,
\qquad
  \singsupp\cal K\bigcap \singsupp \hat v\otimes\hat u=\emptyset.
  \label{suppFKuv-eq}
\end{equation}
Then it holds, with extended action of $\dual{\cdot}{\cdot}$,
\begin{equation}
  \dual{\cal F a(x,D)\cal F^{-1}(\hat u)}{\hat v}
  =\dual{\cal K}{\hat v\otimes\hat u}.
  \label{FaFK-eq}
\end{equation}
\end{thm}

\begin{proof}
For $u\in  D(a(x,D))$ the left-hand side of \eqref{FaFK-eq} makes sense
by \eqref{S'conv-eq}; and the
right-hand side does so by \eqref{suppFKuv-eq}, cf
Section~\ref{exdist-sect}. 
The equality follows from \eqref{S'conv-eq}:

Letting $\psi_m=\psi(2^{-m}\cdot)$ there is some $\nu$ such that
$\psi_\nu=1$ on a neighbourhood of $\supp \psi$, so
$\psi_{m+\nu}\psi_m=\psi_m$ for all $m$. Then $1\otimes\psi_m$ and
$\psi_m(\xi-\eta)\psi_m(\eta)$ equal $1$ on the intersection of the supports
in \eqref{suppFKuv-eq} for all sufficiently large $m$, so
\begin{equation}
  \dual{\cal K}{\hat v\otimes\hat u}
  =\dual{\psi_m(\xi-\eta)\psi_{m+\nu}(\eta)\cal K(\xi,\eta)}{
     \hat v(\xi)\psi_m(\eta)\hat u(\eta)}.
  \label{Kuvpsi-eq}
\end{equation}
Here $\cal K_m=\psi_m(\xi-\eta)\psi_{m+\nu}(\eta)\cal K(\xi,\eta)$ is the
kernel of 
$\cal F A_m\cal F^{-1}$, when $A_m=\OP(a^m(1\otimes\psi_{m+\nu}))$;
cf Proposition~\ref{FKah-prop}.
Clearly $A_m$ has symbol in $S^{-\infty}$. 

Moreover, mollification of $\psi_m\hat u=\cal F u^m$ gives a sequence $(\cal
F u^m)_k$ of functions in $C^\infty_0(\Rn)$, that all have their supports in
a fixed compact set $M$. Invoking regular convergence, 
cf Lemma~\ref{regconv-lem}, it follows that 
\begin{gather}
 \hat v\otimes(\cal F u^m)_k \xrightarrow[k\to\infty]{} 
  \hat v\otimes \cal Fu^m
  \quad\text{in $\cal S'(\R^{2n})$ and $C^\infty(\Omega)$}
   \\
 \R^{2n}\setminus \Omega=\supp\hat v\times \singsupp\cal Fu^m
  =\singsupp(\hat v\otimes\cal Fu^m).
\end{gather}
Since all supports are contained in $\supp\hat v\times M$, 
Theorem~\ref{fuconv-thm} applied on 
$\R^{2n}$ and the continuity of $A_m$ in $\cal S'$ imply
\begin{equation}
  \dual{\cal K_m}{\hat v\otimes(\cal Fu^m)}
  =\lim_{k\to\infty}\dual{\cal K_m}{\hat v\otimes(\cal Fu^m)_k}
  =\dual{\cal F A_m\cal F^{-1}(\cal F u^m)}{\hat  v}.
  \label{Kmumv-eq}
\end{equation}
According to Lemma~\ref{alter-lem} the factor $\psi_{m+\nu}$ can here 
be removed from the symbol of $A_m$, so it is implied by 
\eqref{Kuvpsi-eq}, \eqref{Kmumv-eq} and the
explicit assumption of $\cal S'$-convergence in \eqref{S'conv-eq} that
\begin{equation}
  \dual{\cal K}{\hat v\otimes\hat u}
  =\dual{\cal FA_mu^m}{\hat v}
  =\lim_{m\to\infty}\dual{a^m(x,D)u^m}{\cal F^2 v}
  =\dual{a(x,D)u}{\cal F^2 v},
\end{equation}
since $\cal F^2v\in \cal S$. 
Transposing $\cal F$, formula \eqref{FaFK-eq} results.
\end{proof}

The assumption of $\cal S'$-convergence in \eqref{S'conv-eq} cannot be
omitted from the above proof, although in the last line 
$\dual{a^m(x,D)u^m}{\cal F^2 v}$ is independent of $m$.
Eg $\cos x$ is in $\cal S'(\R)$, but since 
$\cal F(\sum_{j=0}^m
\tfrac{(-1)^j}{(2j)\!}x^{2j})=c_0\delta_0+c_2\delta_0''+
\dots +c_m\delta_0^{(m)}$ the power series converges 
to $\cos x$ in $\cal D'$ but not in
$\cal S'$ as cosine isn't a polynomial.
Moreover, if $\cal F v=1$ around $0$ for some $\cal F v\in
C^\infty_0([-1,1])$, one clearly has
$2\pi=\dual{\sum_{j=0}^m \tfrac{(-1)^j}{(2j)\!}x^{2j}}{\cal F^2v}$ 
for every $m$ as the
derivaties of $\cal Fv$ vanish at the origin. 
And yet $\dual{\cos}{\cal F^2v}=
\dual{\tfrac{1}{2}(\delta_1+\delta_{-1})}{\cal F v}=0\ne 2\pi$.

The next result extends \cite[Prop.~1.4]{JJ05DTL} from
the case of $u\in C^\infty(\Rn)$ with $\supp\hat u\Subset\Rn$ 
to almost arbitrary distributions $u\in D(a(x,D))$;
but the proof is significantly simpler here.
Instead of \eqref{FKsuppu-eq}, the explicit form given in 
\eqref{Xi-eq}--\eqref{Xi-id} is preferred for practical purposes.

\begin{thm}[The spectral support rule]
  \label{supp-thm}
Let $a\in  S^\infty_{1,1}(\Rn\times\Rn)$ and suppose 
$u\in D(a(x,D))$ is 
such that, for some
$\psi\in C^\infty_0(\Rn)$ equalling $1$ around the origin,
the convergence of Definition~\ref{gdef-defn} 
holds in the topology of $\cal S'(\Rn)$, ie
\begin{equation}
  a(x,D)u=\lim_{m\to\infty} a^m(x,D)u^m \quad\text{in}\quad \cal S'(\Rn).
  \label{S'lim-cnd}
\end{equation}
Then \eqref{FKsuppu-eq} holds, that is with $\Xi=\supp\cal K\circ\supp\hat u$
one has
\begin{gather}
   \supp\cal F(a(x,D)u)\subset\overline{\Xi},
 \\
  \Xi=\bigl\{\,\xi+\eta \bigm| (\xi,\eta)\in \supp\hat a,\ 
     \eta\in \supp\hat u \,\bigr\}.
  \label{Sigma-eq}
\end{gather}
When $u\in \cal F^{-1}\cal E'(\Rn)$ then \eqref{S'lim-cnd} 
holds automatically and $\Xi$ is closed for such $u$.
\end{thm}
\begin{proof}
That $\Xi=\supp\cal K\circ\supp\hat u$ has the form in \eqref{Sigma-eq}
follows by substituting $\zeta=\xi+\eta$. 
Using Theorem~\ref{Kuv-thm} instead of Theorem~\ref{AK-thm}, the proof of
Corollary~\ref{AK-cor} can now be repeated mutatis mutandis; which gives the
inclusion in question.

The redundancy of \eqref{S'lim-cnd} for $u\in \cal F^{-1}\cal E'$
follows since, by Lemma~\ref{alter-lem}, one can for large $m$ write 
$a^m(x,D)u^m=\OP(a^m(1\otimes\chi)\psi_m)u$
for a fixed cut-off function $\chi$. Then Proposition~\ref{Stan-prop} 
gives $\cal S'$-convergence, for 
multiplication by $1\otimes\chi$ commutes with $\psi_m(D_x)$
and $a(1\otimes\chi)\in S^{-\infty}$. That $\Xi$ is closed then is
straightforward to verify.
\end{proof}

\begin{rem}
The set $\Xi$ in \eqref{Sigma-eq} 
need not be closed if $\supp\cal Fu$ is non-compact,
for $\supp\hat a$ may contain points arbitrarily
close to the twisted diagonal. 
Eg if $n=1$ and $\supp\hat u=[2,\infty[\,$ whilst
$\supp \hat a$ consists of the $(\xi,\eta)$ such that 
$\eta\ge -\xi-\tfrac{1}{\xi}$ for $\xi\le-1$, 
and $\eta\ge 2$ for $\xi\ge -1$, 
then $(\xi_k,\eta_k)=(-k,k+1/k)$ fulfils 
$\Xi\ni\xi_k+\eta_k=\tfrac{1}{k}\searrow 0$,
although $0\notin \Xi$.
\end{rem}

That $a(x,D)u$ should be in $\cal S'$ in \eqref{S'lim-cnd} is
natural in order that $\cal F a(x,D)u$ makes sense before its support is
investigated. 
One could conjecture that the condition of convergence in $\cal S'$ is
redundant, so that it would suffice 
to assume $a(x,D)u$ is an element of $\cal S'$. 
But it is not clear (whether and) how this can be proved.

\section{Continuity in Sobolev spaces}   \label{cont-sect}
As a last justification of Definition~\ref{gdef-defn} its close connection
to estimates in Sobolev spaces will be indicated.

\subsection{Littlewood--Paley decompositions}   \label{LP-ssect}
For the purposes of this section, one may for a symbol
$a\in S^d_{1,1}(\Rn\times \Rn)$ consider the limit
\begin{equation}
  a_{\psi}(x,D)u=\lim_{m\to\infty }
  \op{OP}(\psi(2^{-m}D_x)a(x,\eta)\psi(2^{-m}\eta))u.
  \label{aPsi-eq}
\end{equation}
By the definition, $u$ is in $D(a(x,D))$ if $a_{\psi}(x,D)u$ exists
for all $\psi$ and is independent of $\psi$,
as $\psi$ runs through $C^\infty_0(\Rn)$ with $\psi=1$ around the origin.

From $a_{\psi}(x,D)$ there is a particularly easy passage to the 
paradifferential decomposition used by J.-M.~Bony \cite{Bon}.
For this purpose,
note that to each fixed $\psi$ there exist $R>r>0$ satisfying
\begin{equation}
  \psi(\xi)=1\quad\text{for}\quad |\xi|\le r;
  \qquad
  \psi(\xi)=0\quad\text{for}\quad |\xi|\ge R\ge 1.
  \label{Rr-eq}
\end{equation}
Moreover it is convenient to let $h$
stand for an integer such that $R\le r2^{h-2}$.
 
To obtain a Littlewood--Paley decomposition from $\psi$, define 
$\varphi=\psi-\psi(2\cdot )$. Then
it is clear that $\varphi(2^{-k}\cdot )$ is supported in a corona,
\begin{equation}
  \supp \varphi(2^{-k}\cdot )\subset \bigl\{\,\xi \bigm| 
   r2^{k-1}\le|\xi|\le R2^k\,\bigr\},
  \qquad \text{for }k\ge 1.
  \label{phi-eq}
\end{equation}
The identity $1=\psi(\xi)+\sum_{k=1}^\infty \varphi(2^{-k}\xi)$ 
follows by letting $m\to\infty $ in the telescopic sum,
\begin{equation}
  \psi(2^{-m}\xi)=\psi(\xi)+\varphi(\xi/2)+\dots+\varphi(\xi/2^m).
  \label{tele-eq}
\end{equation}
Using this one can localise functions $u(x)$ and symbols $a(x,\eta)$ to
frequencies $|\eta|\approx 2^j$ for $j\ge 1$ by setting
\begin{align}
  u_j&=\varphi(2^{-j}D)u, &\qquad
  a_j(x,\eta)&=\varphi(2^{-j}D_x)a(x,\eta)=
  \cal F^{-1}_{\xi\to x}(\varphi(2^{-j}\xi)\hat a(\xi,\eta)).
\\
  u_0&=\psi(D)u, &\qquad
  a_0(x,\eta)&=\psi(D_x)a(x,\eta)
\end{align}
Similarly localisation to balls $|\eta|\le R2^j$ are written, now with upper
indices, as
\begin{equation}
  u^j=\psi(2^{-j}D)u, \qquad
  a^j(x,\eta)=\psi(2^{-j}D_x)a(x,\eta)=
   \cal F^{-1}_{\xi\to x}(\psi(2^{-j}\xi)\cal F_{x\to\xi}a(\xi,\eta)).
\end{equation}
Moreover, $u^0=u_0$ and $a^0=a_0$.
In addition both eg $a_j=0$ and $a^j=0$ should be understood when $j<0$. 
(In order not to have two different meanings of sub- and superscripts on
functions, the dilations $\psi(2^{-j}\cdot )$ are simply written as such;
and the corresponding Fourier multiplier as $\psi(2^{-j}D)$.)
Note that $a^{k}(x,D)=\op{OP}(\psi(2^{-k}D_x)a(x,\eta))$ etc.

However, returning to \eqref{aPsi-eq}, the relation \eqref{tele-eq} 
applies twice, whence bilinearity gives 
\begin{equation}
  a^m(x,D)u^m
=\OP([a_0(x,\eta)+\dots +a_m(x,\eta)][\psi(\eta)+\dots +\varphi(2^{-m}\eta)])u
= \sum_{j,k=0}^m   a_j(x,D)u_k.
  \label{au_bilin-eq}
\end{equation}
Of course the sum may be split in three groups in which $j\le k-h$,
$|j-k|<h$ and $k\le  j-h$, respectively.
In the limit $m\to\infty $ this gives the decomposition
\begin{equation}
  a_{\psi}(x,D)u
  =a_{\psi}^{(1)}(x,D)u+a_{\psi}^{(2)}(x,D)u+a_{\psi}^{(3)}(x,D)u,
  \label{a123-eq}
\end{equation}
whenever $a$ and $u$ fit together such that 
the three series below converge in $\cal D'(\Rn)$:
\begin{align}
    a_{\psi}^{(1)}(x,D)u&=\sum_{k=h}^\infty \sum_{j\le k-h} a_j(x,D)u_k
  =\sum_{k=h}^\infty a^{k-h}(x,D)u_k
  \label{a1-eq}\\
  a_{\psi}^{(2)}(x,D)u&= \sum_{k=0}^\infty
               \bigl(a_{k-h+1}(x,D)u_k+\dots+a_{k-1}(x,D)u_k+a_{k}(x,D)u_k
\notag\\[-2\jot]
   &\qquad\qquad
                +a_{k}(x,D)u_{k-1} +\dots+a_k(x,D)u_{k-h+1}\bigr) 
  \label{a2-eq}\\
   a_{\psi}^{(3)}(x,D)u&=\sum_{j=h}^\infty\sum_{k\le j-h}a_j(x,D)u_k
   =\sum_{j=h}^\infty a_j(x,D)u^{j-h}.
  \label{a3-eq}
\end{align}
Also \eqref{a2-eq} has a brief form, namely
\begin{equation}
  a_{\psi}^{(2)}(x,D)u=\sum_{k=0}^\infty
  ((a^{k}-a^{k-h})(x,D)u_k+a_k(x,D)(u^{k-1}-u^{k-h})).   
\end{equation}
 
One advantage of the decomposition is that the terms of the 
first and last series fulfil a
dyadic corona condition; whereas the in second the spectra are in general
only restricted to balls:

\begin{prop}  \label{corona-prop}
If $a\in S^d_{1,1}(\Rn\times \Rn)$ and $u\in \cal S'(\Rn)$, and 
$r$, $R$ are chosen as in \eqref{Rr-eq} for each auxiliary function $\psi$,
then every $h\in \N$ such that $R\le r2^{h-2}$ gives
\begin{align}
  \supp\cal F(a^{k-h}(x,D)u_k)&\subset
  \bigl\{\,\xi\bigm| 
  \frac{r}{4}2^k\le|\xi|\le \frac{5R}{4} 2^k\,\bigr\}
  \label{supp1-eq}  \\
  \supp\cal F(a_k(x,D)u^{k-h})&\subset
  \bigl\{\,\xi \bigm| 
  \frac{r}{4}2^k\le|\xi|\le \frac{5R}{4} 2^k\,\bigr\}.
  \label{supp3-eq}
\end{align}
Moreover, for $a_{\psi}^{(2)}(x,D)$,
\begin{equation}
  \supp\cal F\big(a_k(x,D)(u^{k-1}-u^{k-h})+(a^k-a^{k-h})(x,D)u_k\big)\subset
  \bigl\{\,\xi\bigm| |\xi|\le 2R 2^k\,\bigr\}
  \label{supp2-eq}
\end{equation}
If $a$ satisfies \eqref{TDC-cnd} this support is contained in 
\begin{equation}
  \bigl\{\,\xi \bigm| \frac{r}{2^{h+1}C} \le |\xi|\le 2R2^k\,\bigr\}
\end{equation}
for all $k\ge h+1+\log_2(C/r)$.
\end{prop}

This proposition follows straightforwardly from the spectral support rule in
Theorem~\ref{supp-thm},
with a special case explained in \cite{JJ05DTL}, so further details 
should hardly be needed here.

In addition one can estimate each series using the
Hardy--Littlewood maximal operator $Mu_k(x)$.
This gives eg for $\nu=1$ in \eqref{a(nu)-est}, when the Fefferman--Stein inequality is used in
the last step,
\begin{equation}
  \bigl(\int_{\Rn} \bigl( \sum_{k=h}^\infty 
    |2^{sk}a^{k-h}(x,D)u_k(x)|^2\bigr)^{\tfrac{p}{2}}\,dx\bigr)^{\tfrac{1}{p}}
  \le c(a)
  \Nrm{\bigl( \sum_{k=0}^\infty |2^{(s+d)k}Mu_k(\cdot )|^2
             \bigr)^{\tfrac{1}{2}}}{p}
  \le c'c(a)\nrm{u}{H^{s+d}_p};
\end{equation}
here $c(a)$ is a continuous seminorm on $a\in S^{d}_{1,1}$. Similar
estimates are obtained for $\nu=2$ and $\nu=3$.
The  reader is referred to \cite{JJ05DTL} for brevity here.
(Although the set-up was more
general there with Besov spaces $B^{s}_{p,q}$ and Triebel--Lizorkin spaces
$F^{s}_{p,q}$, it is easy to specialise to the present 
$H^s_p$-framework, mainly by
setting $q=2$ in the treatment of the $F^{s}_{p,q}$.
One difference in the framework of \cite{JJ05DTL}
is that certain functions $\tilde \Phi_j$ enter the expressions
$a_{j,k}(x,D)u_k$ 
there, but the $\tilde \Phi_j$ 
amount to special choices of $\chi$ in the above formula
\eqref{alter-eq}, hence may be removed when convenient.)
 
Combining such estimates with Proposition~\ref{corona-prop} it follows in a
well-known way that for $\nu=1$ and $\nu=3$,
\begin{equation} \label{a(nu)-est}
  \nrm{a_{\psi}^{(\nu)}(x,D)u}{H^s_{p}}\le  c
  \nrm{u}{H^{s+d}_p}
  \quad\text{for}\quad s\in \R,\ 1<p<\infty .
\end{equation}
For $\nu=2$ this holds for $s>0$,
because the coronas are replaced by balls.

Consequently 
$u\mapsto a_{\psi}(x,D)u
=a_{\psi}^{(1)}(x,D)u+a_{\psi}^{(2)}(x,D)u+a_{\psi}^{(3)}(x,D)u$ 
is a bounded linear operator 
$H^{s+d}_p(\Rn)\to H^s_p(\Rn)$ for $s>0$ and
\begin{equation}
  \nrm{a_{\psi}(x,D)u}{H^s_p}\le C(a)\nrm{u}{H^{s+d}_p},
\end{equation}
where $C(a)$ is a continuous seminorm on $a\in S^d_{1,1}$.
Moreover, if $a$ fulfils 
\eqref{TDC-cnd}, then the last part of
Proposition~\ref{corona-prop} leads to continuity for all $s\in \R$.  

Moreover, density of the Schwartz space $\cal
S(\Rn)$ in $H^{s+d}_p(\Rn)$ yields that $a_{\psi}(x,D)$ is independent of
$\psi$,
for they all agree with $\OP(a)u$ whenever $u\in \cal S(\Rn)$; 
cf \eqref{a123-eq}, \eqref{au_bilin-eq} and \eqref{bilSS-eq}.
So by Definition~\ref{gdef-defn} it follows that $a(x,D)u$ is defined on
every $u\in H^{s+d}_p$ with $s>0$; more precisely one has

\begin{thm}
  \label{Hsp-thm}
Let $a(x,\eta)$ be a symbol in $S^d_{1,1}(\Rn\times \Rn)$.
Then for every $s>0$, $1<p<\infty $ the type $1,1$-operator $a(x,D)$ has
$H^{s+d}_{p}(\Rn)$ in its domain and it is
a continuous linear map
\begin{equation}
  a(x,D)\colon H^{s+d}_p(\Rn)\to H^s_p(\Rn).
\end{equation}
This property extends to all $s\in \R$ when $a$ fulfils the twisted diagonal
condition \eqref{TDC-cnd}.
\end{thm}

In \cite{JJ05DTL} a similar proof was given for Besov and Lizorkin--Triebel
spaces, ie for $B^{s}_{p,q}$ and $F^{s}_{p,q}$. But this contains the above
Theorem~\ref{Hsp-thm} in view of the well-known identification
$H^s_p(\Rn)=F^{s}_{p,2}(\Rn)$ for $1<p<\infty $, which through a reduction
to $s=0$ results from the Littlewood--Paley inequality.
The reader is referred to the more general continuity
results in \cite{JJ05DTL}, which also cover the H{\"o}lder--Zygmund classes
because of the identification $C^s=B^s_{\infty ,\infty }$.

However, a little precaution is needed because 
$\cal S(\Rn)$ is not dense in $B^s_{\infty,q}$. 
Even so $a(x,D)$ is defined on and bounded
from $B^{s}_{\infty ,q}$ for $s>d$ (and $s=d$, $q=1$ cf \eqref{Fdp1-eq} ff
and \cite[(1.6)]{JJ05DTL}), which may be seen from the Besov space estimates
of \cite{JJ05DTL} and the argument preceeding Theorem~\ref{Hsp-thm}  
as follows. 
By lowering $s$ one can arrange that $q<\infty $,
in which case it is well-known that $B^s_{\infty ,q}\bigcap\cal F^{-1}\cal
E'$ is dense; whence $a_{\psi}(x,D)u$ is independent of $\psi$ for all 
$u\in B^{s}_{\infty ,q}$ if it is so for all $u\in \cal F^{-1}\cal E'$. 
This last property is a consequence of the fact that $\cal F^{-1}\cal E'$
is in the domain of $a(x,D)$; cf Theorem~\ref{consistency-thm} and
Remark~\ref{consistency-rem}.

\begin{rem}   \label{iden-rem}
It is evident that the counter-example in Proposition~\ref{WF-prop} relied
on an 
extension of continuity of $a_{2\theta}(x,D)$ to a bounded operator
$H^{s+d}\to H^s$ for arbitrary $s<d$. Moreover, this extension has not
previously been identified with the definition of $a_{2\theta}(x,D)$ by
vanishing frequency modulation. However, by the density of $\cal S$, it
follows from the last part of Theorem~\ref{Hsp-thm}  that these
two extensions are identical, whence the operators in
Definition~\ref{gdef-defn} lack the microlocal property in the treated cases.
\end{rem}

\subsection{Composite functions}   \label{composite-ssect}
Finally it is verified that the formal definition of type $1,1$-operators by
vanishing frequency modulation also plays well together with Y.~Meyer's 
formula for composite functions.

\bigskip

Consider the map $u\mapsto F\circ u$ given by $F(u(x))$ for a fixed $F\in
C^\infty (\R)$ and a real-valued $u\in H^{s_0}_{p_0}(\Rn)$ for
$s_0>n/{p_0}$, $1<p_0<\infty $.
Then $u$ is uniformly continuous and bounded on $\Rn$ as well as in
$L_p(\Rn)$ for $p_0\le p\le \infty $.
Note that with the notation of the previous section, and in particular  
\eqref{tele-eq}, one has in $L_p(\Rn)$ 
\begin{equation}
  u_0+u_1+\dots +u_m=u^m=2^{mn}\cal F^{-1}\psi(2^m\cdot )*u
  \xrightarrow[m\to\infty ]{~} u.
\end{equation}
Assuming that $F(0)=0$ when $p<\infty $, then  
$v\mapsto F\circ v$ is Lipschitz continuous on 
the metric subspace $L_p(\Rn,B)$ for every ball $B\Subset \R$,
\begin{gather}
  F(w(x))-F(v(x))=\int_0^1 F'(v(x)+t(w(x)-v(x)))\,dt\cdot (w(x)-v(x))  
\\
  \nrm{F\circ w-F\circ v}{p}\le \sup_{B}|F'|\cdot \nrm{w-v}{p}.
\end{gather}
Since $\nrm{u^m}{\infty }\le \nrm{\cal F^{-1}\psi}{1}\nrm{u}{\infty }$, 
one can take $B$ so large that it contains $u(\Rn)$ and $u^m(\Rn)$ for every
$m$, so since $u^k-u^{k-1}=u_k$ 
it follows
from the Lipschitz continuity  
that with limits in $L_p$, $p_0\le p\le \infty $, 
\begin{equation}
  \begin{split}
  F(u(x))&=\lim_{m\to\infty } F(u^m(x))
  =F(0)+\lim_{m\to\infty } \sum_{k=0}^m(F(u^k(x))-F(u^{k-1}(x)))
\\
  &= F(0)+\sum_{k=0}^\infty 
   \int_0^1 F'(u^{k-1}(x)+tu_{k}(x))\,dt\cdot \varphi(2^{-k}D)u(x).
  \end{split}
  \label{Flin-eq}
\end{equation}
Setting $m_k(x)=\int_0^1 F'(u^{k-1}(x)+tu_{k}(x))\,dt$
it is not difficult to see that $m_k\in C^\infty (\Rn)$ 
with bounded derivatives of any order because 
$D^\beta u_k=2^{k(n+|\beta|)}D^\beta \check \varphi(2^k\cdot )*u$; and since 
$2^k\approx |\eta|$ on $\supp\varphi(2^{-k}\cdot )$ that
\begin{equation}
  a_u(x,\eta):= \sum_{k=0}^\infty m_k(x)\varphi(2^{-k}\eta)
 \in   S^0_{1,1}(\Rn\times \Rn).
  \label{auxD-eq}
\end{equation}
That the corresponding type $1,1$-operator $a_u(x,D)$ gives important
information on the function $F(u(x))$ was the idea of 
Y.~Meyer \cite{Mey80,Mey81},
and it is now confirmed that his results remain valid when the operators are
based on Definition~\ref{gdef-defn}:
 
\begin{thm}   \label{composite-thm}
When $u\in H^{s_0}_{p_0}(\Rn)$ for $s_0>n/{p_0}$, $1<p_0<\infty$ 
is real valued, 
and $F\in C^\infty (\R)$ with $F(0)=0$, 
then the type $1,1$-operator $a_u(x,D)$ is a
bounded linear operator
\begin{equation}
  a_u(x,D)\colon  H^s_p(\Rn)\to H^s_p(\Rn)
  \quad\text{for every}\quad s>0,\ 1<p<\infty .
\end{equation} 
Taking $s=s_0$, $p=p_0$ one has $a_u(x,D)u(x)=F(u(x))$, and the map
$u\mapsto  F\circ u$ is continuous on $H^{s_0}_{p_0}(\Rn,\R)$.
\end{thm}
\begin{proof}
The continuity on $H^s_p$ follows from Theorem~\ref{Hsp-thm} since $a_u\in
S^0_{1,1}$. As the proof of this theorem shows, the operator norm
$\nrm{b(x,D)}{}$ in $\B(H^s_{p})$ is estimated by a seminorm $c(b)$ on $b\in
S^0_{1,1}\subset S^1_{1,1}$, and \eqref{auxD-eq} converges in $S^1_{1,1}$, 
so one has in $\B(H^s_p)$ that
\begin{equation}
  \OP(\sum_{k=0}^m m_k(x)\varphi(2^{-k}\eta)) 
  \xrightarrow[m\to\infty ]{~}  a_u(x,D).
\end{equation}
By \eqref{Flin-eq} this implies that in the larger space $L_{p_0}$
\begin{equation}
  a_u(x,D)u=
  \lim_{m\to\infty }\OP(\sum_{k=0}^m m_k(x)\varphi(2^{-k}\eta))u=
  \sum_{k=0}^\infty m_k(x)\varphi(2^{-k}D)u= F\circ u.
\end{equation}
Hence $u\mapsto F\circ u$ is a map $H^{s_0}_{p_0}\to H^{s_0}_{p_0}$, which
is continuous since for $v\to u$
\begin{equation}
  \begin{split}
  F(v(x))-F(u(x))=a_{u}(x,D)(v-u)+&[a_v(x,D)-a_u(x,D)]u
\\
                  +&[a_v(x,D)-a_u(x,D)](v-u)  \to 0.
  \end{split}
\end{equation}
Indeed, by continuity of $a_u(x,D)$ the first term tends to $0$, and by the
Banach--Steinhauss theorem the two other terms do so if only $a_v(x,D)\to
a_u(x,D)$ in $\B(H^{s_0}_{p_0})$, ie if $a_v\to a_u$ in
$S^0_{1,1}$. However, the non-linear map $u\mapsto a_u$ is continuous from
$H^{s_0}_{p_0}$ to $S^0_{1,1}$, for
$(1+|\eta|)^{|\alpha|-|\beta|}
|D^\alpha_\eta D^\beta_x(a_v(x,\eta)-a_u(x,\eta))|$
is at each $\eta$ estimated uniformly by terms that may have 
$\nrm{v-u}{\infty }$ as a factor or contains
$\sup_{x\in \Rn}\int_0^1|F^{(l)}(v^{k-1}(x)+tv_k(x))-
                         F^{(l)}(u^{k-1}(x)+tu_k(x))|\,dt$, 
which tends to $0$ by the
uniform continuity of $F^{(l)}$ on a sufficiently large ball.
\end{proof}

Among the merits of the theorem, note that for non-integer $s$
it is non-trivial to prove that $F(u(x))$ is in $H^s_p$ when $u$ is so.
When needed the reader may derive similar results for the $B^{s}_{p,q}$ and
$F^{s}_{p,q}$ from the estimates in \cite{JJ05DTL}.
Moreover, continuity of $u\mapsto F\circ u$ is as shown a straightforward
consequence of the factorisation $a_u(x,D)u$, but this was not mentioned in 
\cite{Mey80,Mey81,H97}.

\begin{rem}
As a small extension of the above, it may be noted that when $F'$ is bounded
on $\R$, then the assumption on $u$ can be relaxed to $u\in L_{p_0}$ for
$1\le p_0\le \infty $, for $F(u(x))$ is defined, and
the linearisation formula \eqref{Flin-eq} 
still holds as $u\mapsto F\circ u$ is Lipschitz continous on 
$L_{p_0}(\Rn,\R)$ in this case (however, the symbol $a_u(x,\eta)$ has much
weaker properties). 
\end{rem}

\subsection*{Acknowledgement}
My thanks are due to the anonymous referee for requesting
a more explicit comparison with the existing literature.

%
%
\providecommand{\bysame}{\leavevmode\hbox to3em{\hrulefill}\thinspace}
\providecommand{\MR}{\relax\ifhmode\unskip\space\fi MR }
\providecommand{\MRhref}[2]{%
  \href{http://www.ams.org/mathscinet-getitem?mr=#1}{#2}
}
\providecommand{\href}[2]{#2}

\end{document}